\documentclass{conm-p-l}
\usepackage{overpic}
\usepackage{pifont}
\usepackage{fancybox}
\usepackage{bm}

\newtheorem{theorem}{Theorem}[section]
\newtheorem{lemma}[theorem]{Lemma}
\newtheorem{proposition}[theorem]{Proposition}
 
\theoremstyle{definition}

\newtheorem{remark}[theorem]{Remark}
\newtheorem{assumption}[theorem]{Assumption}

\numberwithin{equation}{section} 
\newcommand{\field}[1]{\mathbb{#1}}
\newcommand{\R}{\field{R}} 
\newcommand{\Z}{\field{Z}}
\newcommand{\N}{\field{N}}
\newcommand{\C}{\field{C}}
\renewcommand{\P}{\field{P}}

\renewcommand{\AA}{{\mathcal A}}

\newcommand{\mb}[1]{\bm{#1}}

\newcommand{\cp}{\mathrm{cap}}

\newcommand{\Int}{\mathop{\rm Int}}
\newcommand{\Ext}{\mathop{\rm Ext}}

\renewcommand{\Re}{\mathop{\rm Re}}
\renewcommand{\Im}{\mathop{\rm Im}}
\newcommand{\dist}{\mathop{\rm dist}}
\newcommand{\sm}{\setminus }

\newcommand{\capto}{\stackrel{\cp}{\longrightarrow}}

\begin{document}

\title[A John Nuttall's  work 25 years later]{Heine, Hilbert, Pad\'e, Riemann, and Stieltjes: a John Nuttall's  work 25 years later.}

\author[A.~Mart\'{\i}nez-Finkelshtein]{Andrei Mart\'{\i}nez-Finkelshtein}
\address{Department of Statistics and Applied Mathematics
University of Almer\'{\i}a, SPAIN, and
Instituto Carlos I de F\'{\i}sica Te\'{o}rica y Computacional,
Granada University, SPAIN}
\email{andrei@ual.es}
\thanks{The first author was partially supported by Junta de Andaluc\'{\i}a, grant FQM-229 and the Excellence Research Grant P09-FQM-4643, as well as by the
research projects MTM2008-06689-C02-01 and MTM2011-28952-C02-01 from the Ministry of Science and Innovation of Spain and the European Regional Development Fund (ERDF)}

\author[E.A.~Rakhmanov]{Evgenii A.~Rakhmanov}
\address{Department of Mathematics,
University of South Florida, USA 
}
\email{rakhmano@mail.usf.edu}

\author[S.P.~Suetin]{Sergey P.~Suetin}
\address{V.~A.~Steklov Mathematical Institute of the Russian Academy of Sciences,
Moscow, Russia
}
\email{suetin@mi.ras.ru}
\thanks{The third author was partially supported by  the Russian Fund for Fundamental Research grant 11-01-00330, and the program ``Leading Scientific Schools of the Russian Federation'', grant NSh-8033.2010.1.}

\date{February 3, 2009.}

\dedicatory{This paper is dedicated to the 60th Birthday of Francisco (Paco) Marcell\'an}


\begin{abstract}
In 1986 J.\ Nuttall published in Constructive Approximation the paper \cite{MR891770}, where with his usual insight he studied the behavior of the denominators (``generalized Jacobi polynomials'') and the remainders of the Pad\'e approximants to a special class of algebraic functions with 3 branch points. 25 years later we try to look at this problem from a modern perspective. 
On one hand, the generalized Jacobi polynomials constitute an instance of the so-called Heine-Stieltjes polynomials, i.e.~they are solutions of linear ODE with polynomial coefficients. On the other, they satisfy complex orthogonality relations, and thus are suitable for the Riemann-Hilbert asymptotic analysis. Along with the names mentioned in the title, this paper features also a special appearance by Riemann surfaces, quadratic differentials, compact sets of minimal capacity, special functions and other characters.
\end{abstract}

\maketitle


\section{Pad\'e approximants to algebraic functions}

John Nuttall, whose name appears in the title along with such distinguished, actually illustrious, colleagues,  initiated the study of convergence of Pad\'e approximants for multivalued analytic functions on the plane. Obviously, he was not the first to consider this problem; the best known result in this sense is a theorem of Markov (or Markoff) \cite{MR1554864}, see also \cite{MR1130396}, which assures the \emph{locally uniform} convergence of diagonal Pad\'e approximants to Markov functions: 
if 
$$
\widehat \sigma(z):=\int \frac{d\sigma(t)}{z-t},
$$
where $\sigma$ is a positive measure compactly supported on $\R$ with an infinite number of points of increase, then the diagonal Pad\'e approximants $[n/n]_{\widehat \sigma}$ (see the definition in Section \ref{sec:statement}) to $\widehat \sigma$, which coincide with the approximants of the Chebyshev or $J$-continued fraction for this function, converge to $\widehat \sigma$ uniformly on compact subsets of the complement to the convex hull of the support of $\sigma$, and the convergence holds with a geometric rate. This theorem applies in particular to functions as
$$
\frac{1}{(z^2-1)^{1/2}} \quad \text{or} \quad (z^2-1)^{1/2} - z.
$$
In the same vein, Dumas \cite{Dumas:1908fk} studied the case of the function of the form
$$
f(z)=\left( (z-a_1)(z-a_2)(z-a_3)(z-a_4)\right)^{1/2}- z^2 + \frac{a_1+a_2+a_3+a_4}{2}\, z , 
$$
with points $a_j\in \C$ in general position and the branch of the square root selected in such a way that $f$ is bounded at infinity. Dumas 
observed that the poles of the Pad\'e approximants to $f$ can be dense in $\C$. 

However, Nuttall was the first to abandon the real line completely and start a convergence theory in a truly complex situation. From the Dumas' work it was clear that in a general situation we can no longer expect  uniform convergence\footnote{We cannot expect uniform convergence even along subsequences, as it was shown in \cite{MR1882132, MR1957951, MR1983783}.}. The appropriate notion is the convergence   \emph{in capacity} \cite{Ransford95, Saff:97, MR93d:42029}, that was developed independently by Gonchar and Nuttall. This is an analogue of convergence in measure, where the Lebesgue or plane measure is replaced by the logarithmic capacity of the set.

Still, the question about the domain of convergence (even in capacity) remained: if the approximated function $f$ has a multi-valued analytic continuation to  $\C$ except for a finite number of branch points, then the single-valued Pad\'e approximants $[n/n]_f$ cannot converge to $f$ in this whole domain. They must ``choose'' the appropriate region of convergence where $f$ is single-valued too, and the boundary of this region should attract a sufficient number of poles of $[n/n]_f$. 

\begin{figure}[htb]
\centering \begin{tabular}{ll} \hspace{0cm}\mbox{\begin{overpic}[scale=0.6]%
{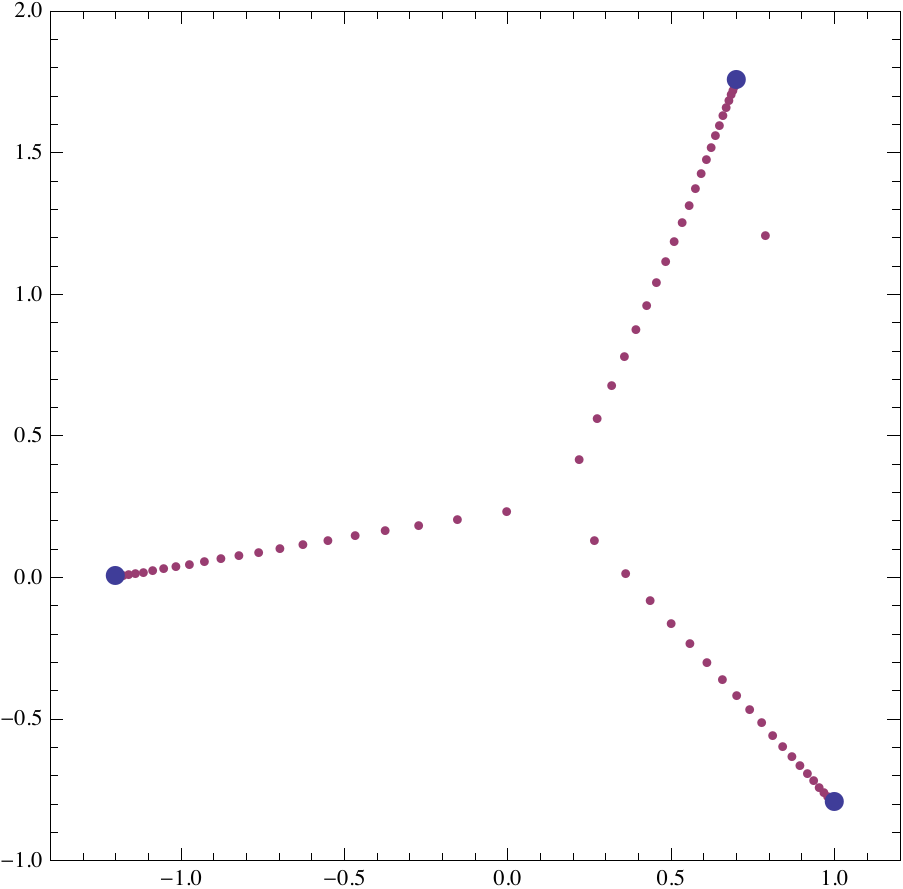}%
\end{overpic}} &
\mbox{\begin{overpic}[scale=0.6]%
{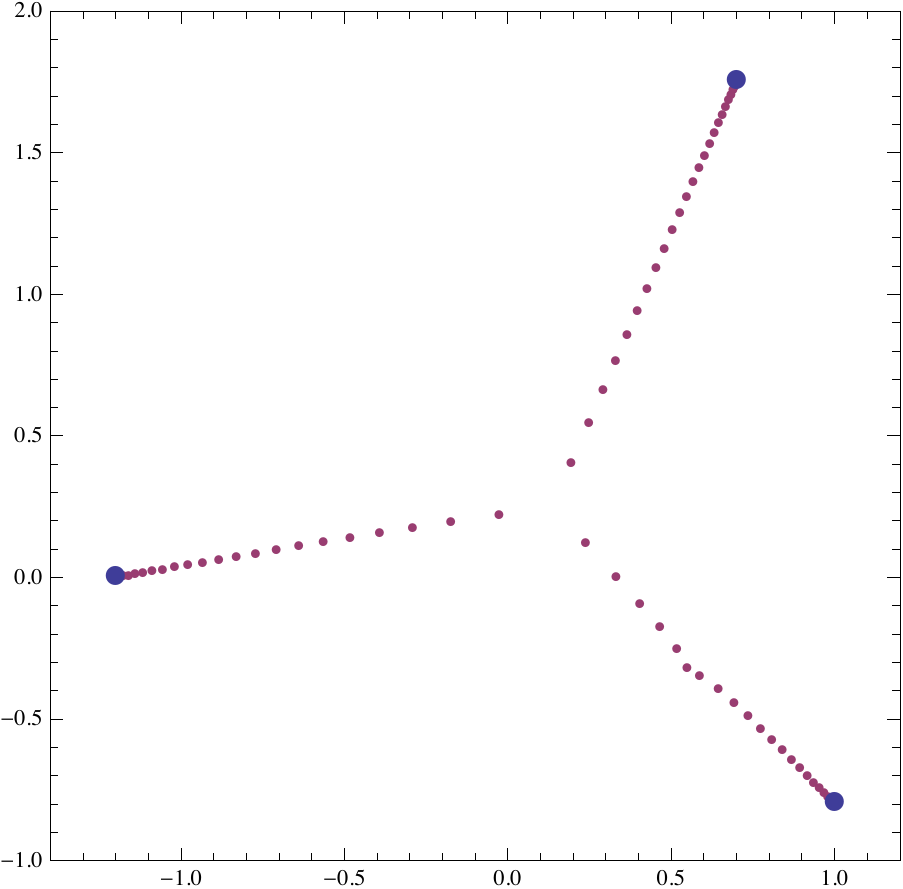}%
\end{overpic}}
\end{tabular}
\caption{Poles of $\pi_{71}$ and $\pi_{72}$ for $f(z)=(z+1.2)^{-3/7}(z-0.7-1.75 i)^{1/7}(z-1-0.8 i)^{2/7}$. Clearly visible is a ``spurious'' pole (left) or a distortion of the location of the poles (right).}
\label{fig:polesPade}
\end{figure}

In \cite{MR0613842, MR769985, MR0487173} Nuttall generalized Markov's theorem by considering a class of hyperelliptic functions of the form $r_1+ r_2 h^{-1/2}$, where $h$ is a polynomial of even degree and simple poles, and $r_j$ are holomorphic functions (it was extended later to meromorphic functions in the work of Stahl \cite{MR98g:41018} and Suetin \cite{MR1805599}). For these functions he found the domain where the convergence takes place: it is a complement to a system of arcs determined by the location of the branch points. Nuttall characterized this set  as having a minimal logarithmic capacity among all other systems of cuts making the approximated function single-valued in their complement (see e.g.\ Figure~\ref{fig:polesPade}).  In \cite{MR0613842} Nuttall conjectured also that this result is valid for any analytic function on $\C$ with a finite number of  branch points.

The complete proof of this conjecture was given, even in a greater generality, by H.~Stahl in a series of papers \cite{MR88d:30004a, MR88d:30004b, Stahl:86, MR90i:30063, MR99a:41017}, under the only assumption that the singularities of the function $f$ form a \emph{polar set}, i.e.~a set of logarithmic capacity zero, see Theorem \ref{thm:stahl} below. Stahl also characterized the analytic arcs forming the boundary of this domain as trajectories of a rational quadratic differential with poles at the singularities of $f$. They are also a case of the so-called Boutroux curves, see e.g.~\cite{Bertola2007}.


The general results of Nuttall and Stahl (and also of Gonchar and Rakhmanov \cite{Gonchar:84, Gonchar:87}) deal essentially with convergence in capacity and weak (equivalently, $n$-th root) asymptotics of the denominators and residues of the Pad\'e approximants. However, strong or Szeg\H{o}-type asymptotics is extremely interesting, at least in order to clarify the behavior of the spurious poles. As it was mentioned, 
poles that appear within the domain of convergence in capacity and that receive the name (coined by G.\ Baker in the 1960s) of \emph{spurious}, \emph{floating} or \emph{wandering} poles \cite{MR99k:41019, MR1805599}, can become the main obstacle for the uniform convergence. For some classes of elliptic and hyperelliptic functions \cite{MR2553077, MR1805599, MR2042915}, the dynamics of the spurious poles is completely determined by the properties of the Riemann surface underlying the approximated function. Moreover, in the elliptic case \cite{MR2042915} there is only one wandering pole, which greatly simplifies the description of the asymptotic behavior of the Pad\'e approximants.

The results in \cite{MR2042915} were obtained when the approximated functions could be represented as Cauchy integrals supported on the critical trajectory of certain quadratic differential with four poles when this trajectory consists of two disconnected components; the analogue for the Chebotarev set was treated in \cite{Baratchart:2011fk}.  

In this paper we analyze the strong asymptotics of the Pad\'e denominators and the residues of the Pad\'e approximants for a canonical class of algebraic functions -- a generalization of the Jacobi weight on $[-1,1]$, revisiting and extending the results of J.~Nuttall in \cite{MR891770}. We look at this problems from two different perspectives, which give us two formally distinct answers. One of the main goals is to understand the relation between these two asymptotic expressions, in order to get new insight into the nature of this problem and the methods we have used.

At the final stage of the preparation of this manuscript we learned about the closely related work \cite{Aptekarev:2011fk} where the authors also apply one of the techniques used by us to the asymptotic analysis of the Pad\'e denominators in a slightly more general situation. Our works, although close in spirit, have a number of substantial differences, and thus are rather complementary than overlapping.

\section{Statement of the problem} \label{sec:statement}

Let $a_j$, $j=1, \dots, p$, be distinct and in general, non-collinear points on the complex plane $\C$, $\AA:=\{a_1, \dots a_p\}$.
Let $\mathfrak A(\overline\C\sm \mathcal A)$ denote the set of functions $f$ holomorphic at infinity and such that $f$ can be continued analytically (as a multivalued function) to the whole $\overline\C\sm \mathcal A$. 
A \emph{diagonal Pad\'e approximant} to $f\in \mathfrak A(\overline\C\sm \mathcal A)$ is a rational function $\pi_n=[n/n]_f=P_n/Q_n$ of type $(n,n)$, that is, such that both $P_n, Q_n \in \P_n$ (where $\P_n$ denotes the class of algebraic polynomials of degree $\leq n$), which has a maximal order of contact with $f$ at infinity: 
\begin{equation}
\label{defPadenonlinear}
f(z)-\pi_n(z) = \mathcal{O}\left(1/z^{2n+1}\right) \quad \mbox{as} \quad z\to\infty.
\end{equation}
This condition may be impossible to satisfy, but following Frobenius, we can obtain the coefficients of $ P_n$ and $Q_n$ as a solution to the linear system
\begin{equation}
\label{defPadelinear}
R_n(z):=Q_n(z)f(z)-P_n(z) = \mathcal{O}\left(1/z^{n+1}\right) \quad \mbox{as} \quad z\to\infty, \quad Q_n\not \equiv 0.
\end{equation}
Equations \eqref{defPadelinear} form an undetermined homogeneous linear system. Although the solution of \eqref{defPadelinear} is not unique, the Pad\'e approximant (rational function) $\pi_n=P_n/Q_n$ is.  Hereafter, $(P_n,Q_n)$ will always stand for the unique pair of relatively prime polynomials determining $\pi_n$, and $Q_n$ is taken monic. The degree of $Q_n$ could be strictly $< n$. Nevertheless, in the beginning we assume that $n$ is a normal index, so that $Q_n$ is of degree exactly $n$ and \eqref{defPadenonlinear} holds.

The tight connection of the analytic theory of Pad\'e approximants with the (complex) orthogonal polynomials is given by the fact that the denominators $Q_n$ satisfy an orthogonality relation
\begin{equation}
\label{orthogonality1}
\oint  t^k Q_n(t) f(t)\, dt =0, \quad k=0, 1, \dots, n-1,
\end{equation}
where we integrate along a closed Jordan curve encircling $\mathcal A$. 

Let $\mathfrak K_f$ be the family of compacts $K$ containing $\mathcal A$ and such that $f$ has a holomorphic (single-value) continuation to $\overline\C\sm K$. 
From the fundamental work of Stahl \cite{Stahl:86,MR88k:41013, MR90f:40003, MR95a:42034} it follows that there exists $\Gamma \in \mathfrak K_f$ of minimal capacity, called the \emph{Stahl's compact}. It is comprised of a finite number of analytic curves that are trajectories of a closed quadratic differential (see \cite{MR2770010} or \cite{Strebel:84}) and has a connected complement in $\C$. In the particular case of $p=3$ this is a star-shaped compact set known as the \emph{Chebotarev compact}, see Section \ref{sec:Chebotarev}.
\begin{theorem}[Stahl]\label{thm:stahl}
(i) There exists a polynomial 
$$
V(z)=\prod_{k=1}^{p-2} (z-v_k)
$$ 
such that the complex Green function for $\Gamma$ 
is
\begin{equation}
\label{defGreenF}
G(z,\infty):=\int_{a_1}^z \sqrt{\frac{V(t)}{A(t)}}\, dt, \quad A(z):=\prod_{j=1}^p (z-a_j),\quad z\in \C\setminus \Gamma. 
\end{equation}
In other words, $\Gamma$ is given by the level curves of $\Re G(\cdot,\infty)=0$ that join zeros of $A$ or $V$. Alternatively, $\Gamma$ is made of the closure of \emph{critical trajectories} of the quadratic differential $-(V/A)(z) dz^2$.

(ii) The normalized zero counting measure for $Q_n$ converges weakly to the \emph{equilibrium measure} $\lambda_\Gamma$ of $\Gamma$, given by 
\begin{equation}
\label{equilibriumMeasureGeneral}
d\lambda_\Gamma(z)=\frac{1}{\pi i}\, \sqrt{\frac{V(z)}{A(z)}} dz,
\end{equation}
with an appropriate choice of the branch of the square root.

(ii) With an appropriate normalization of $R_n$, 
$$
\frac{1}{n}\, \log|R_n(z)| \capto - \Re G(z,\infty), \quad n\to \infty,
$$
where   $\capto$  denotes convergence in capacity  in $\C\setminus \Gamma$.
\end{theorem}
\begin{remark}
Depending on the function $f$ (and the corresponding class $\mathfrak K_f$), Stahl's compact is not necessarily connected. For instance, already for $p=4$ it can be a tree (for $f(z)=A^{1/4}(z)$) or a union of two analytic arcs (when $f(z)=A^{1/2}(z)$). In the class $\mathfrak K_f$ Stahl's compact however is completely characterized by its \emph{$S$-property}, namely
$$
\frac{\partial}{\partial n_-} \left( \Re G(z, \infty) \right)= \frac{\partial}{\partial n_+} \left( \Re G(z, \infty) \right), \quad z\in \Gamma^\circ,
$$
where $n_\pm$ are the normal vectors to $\Gamma^\circ$. The $S$-property and its generalizations play a crucial role in many branches of analysis and mathematical physics, see e.g.~\cite{Aptekarev:2011fk, Bertola2007, Gonchar:87, MR2770010, MFRS}.
\end{remark}
 
In this paper we concentrate on a canonical example of a function from $\mathfrak A(\overline\C\sm \mathcal A)$. Namely, let   $\alpha_j \in \R\setminus \Z$, $j=1, \dots, p$, be such that  $\alpha_1+\dots+\alpha_p=0$. We additionally assume that no proper subset of $\alpha_j$'s adds up to an integer, so that the corresponding Stahl's compact $\Gamma \in \mathfrak K_f$ is a continuum. This is a sufficient condition for an underlying Riemann surface having the maximal genus (see below). However, for more precise hypotheses, see Assumption \ref{assumption1}. The simplest non-trivial example of this situation is when $p=3$, when $\Gamma$ is star-shaped (Chebotarev continuum). As it follows from the works of Stahl, such ``stars'' along with analytic curves are the main building blocks for a generic $\Gamma$. 

Let
\begin{equation}
\label{f}
f(z)=\prod_{j=1}^p (z-a_j)^{\alpha_j} =1+\sum_{k=1}^\infty \frac{f_k}{z^k},
\end{equation}
where the expansion is convergent in the neighborhood $\mathfrak O:=\{z\in \overline\C:\, |z|>\max_j |a_j| \}$ of infinity. We will agree in denoting by $f^{1/2}$ the branch of the square root in $\mathfrak O$ such that $f^{1/2}(\infty)=1$.

As it was mentioned above, our main goal is to find the strong asymptotics of the Pad\'e denominators $Q_n$, as $n\to \infty$. We derive this asymptotics using two complementary methods. The first one, developed in Section \ref{sec:wkb}, is based on the differential equation satisfied by $f$ and is a combination of the original ideas of Nuttall from \cite{MR891770} with some new developments in the asymptotic theory of generalized Heun differential equations. The second method  is the non-linear steepest descent analysis of Deift and Zhou (see e.g.~\cite{MR2000g:47048}) based on the matrix Riemann-Hilbert problem \cite{Fokas92} solvable in terms of $Q_n$ and $R_n$. In Section \ref{sec:RH} we apply it exclusively to the case of $p=3$.

These two methods provide formally different expressions for the leading term of asymptotics of the Pad\'e denominators and Pad\'e residues. We find their comparison in Section \ref{sec:comparison} very illuminating.

The form of the asymptotics for $Q_n$ and $R_n$ was actually conjectured by Nuttall in \cite{MR769985} in terms of a function solving certain \emph{scalar} boundary value problem. We show also that our results match the Nuttall's conjecture, see Section \ref{sec:comparison}.

\section{Heine and Stieltjes, or asymptotic analysis based on the Liuoville-Green approximation} \label{sec:wkb}

The key observation is that function $f$ in \eqref{f} is semiclassical: it satisfies the ODE
$$
\frac{f'(z)}{f(z)}=\sum_{j=1}^p \frac{\alpha_j}{z-a_j}= \frac{B}{A}(z),\quad A(z):=\prod_{j=1}^p (z-a_j), \quad B\in \P_{n-2}.
$$
It can be proved by standard methods (see e.g.~\cite{MR891770}) that as a consequence, the Pad\'e denominators $Q_n$, the Pad\'e numerators $P_n$, and the remainders $R_n$ satisfy the Laguerre equations:
\begin{theorem}
\label{theorem1}
For each normal index $n$ there exist polynomials $h_n(x)=x^{p-2}+\dots\in \P_{p-2}$ and $D_n(x)=x^{2p-4}+\dots \in \P_{2p-4}$, such that
\begin{equation}
\label{1}
A h_n y'' + \left(A' h_n - A h_n' - B h_n \right) y' -n(n+1) D_n y =0
\end{equation}
is solved by $R_n$,  $Q_n f$ and $P_n$.
\end{theorem}
\begin{remark}
In the case of $Q_n$, the ODE is of the form
$$
A h_n y_n'' + \left(A' h_n - A h_n' + B h_n \right) y_n' -n(n+1) D_n y_n =0.
$$
\end{remark}
Let us use the notation $z_{k,n}$ for the zeros of the polynomials $h_n$ from Theorem \ref{theorem1}:
\begin{equation}
\label{defzeroshn}
h_n(z)=\prod_{k=1}^{p-2}(z-z_{k,n}).
\end{equation}
In order to simplify the situation and concentrate on the main ideas we impose the following assumptions on the zeros of $h_n$:
\begin{assumption}
\label{assumption1}
There exists a constant $M>0$ such that for all sufficiently large $n$,
\begin{equation}
\label{mainassumption}
|z_{k,n}|\leq M \quad \text{and} \quad |A V_n h_n'(z_{k,n})|\geq C, \quad k=1, \dots, p-2.
\end{equation}
In other words, all zeros $z_{k,n}$ of $h_n$ belong to the disk $|z|\leq M$, and they stay away from the zeros of $A V_n$ and from each other.\footnote{ The general case requires the spherical normalization for $h_n$ and has to be treated separately. We avoid further discussion of  this situation for the sake of simplicity.}
\end{assumption}
Observe that the second part of this assumption is completely innocent: in a general case, we will have a reduction in genus and in a number of cycles in the basis of the underlying Riemann surface, see below.

From Assumption \ref{assumption1} it follows that for all sufficiently large $n$ the zeros of $D_n$  lie in the disk $|z|\leq 2M$, see e.g.~\cite{MR2770010}, so that that the set
$$
e_n:=\{z\in \C:\, A h_n D_n=0\}
$$
is uniformly bounded.

Define in $\mathfrak O$,
$$
H_n(z):=\int_{a_1}^z \sqrt{\frac{D_n}{A h_n}(t)}\, dt,
$$
where the branch is chosen such that  $H_n(z)=\log z +\mathcal O(1)$ as $z\to \infty$. It can be extended as an analytic and multivalued function to the whole $\overline \C$.

The polynomial solution $P_n$ of \eqref{1} is known as a  \emph{Heine-Stieltjes polynomial}, while the corresponding coefficient $D_n$ in \eqref{1} is called a \emph{Van Vleck polynomial}, see e.g.~\cite{MR2770010, MR2647571}. Theorem 2.1 from \cite{MR2647571} gives a global description of the trajectories of the quadratic differential $(H_n')^2 (z) dz^2$. In particular, it is a quasi-closed differential with one trajectory emanating from each zero of $A$ and ending at infinity. 
Combining techniques from  \cite{MR2647571} and \cite{MR891770} we get
\begin{theorem}\label{theorem2}
For any $a\in \AA$ there exists a progressive path\footnote{Progressive path $\gamma$ means that $\Re H_n(z)$ is non increasing along $\gamma$.} $\gamma=\gamma(a)$, starting at a point $z_0\in \mathfrak O$ and returning back to $z_0$, which is homotopic in $\C\setminus e_n$ to a contour $\widetilde \gamma$ with $a\in \Int(\widetilde \gamma)$ and $e_n\setminus \{a\} \subset \Ext(\widetilde \gamma)$.

For any such a progressive path $\gamma$  we have for $z \in \gamma \cup \mathfrak O$,
\begin{align}\label{2}
R_n(z) = & C_{n,1}\, \frac{h_n^{3/4} f^{1/2}}{(A D_n)^{1/4}}(z)\, e^{-(n+1/2) H_n(z)} \, \left( 1+\delta_1(z) \right). 
\end{align}
If for $\rho >0$, $\dist(z, e_n)\geq \rho$, then $n |\delta_1(z)|$ is uniformly bounded by a constant depending on $\rho$. 
\end{theorem}
\begin{remark}
This formula should be understood in the following way: the right hand side is chosen for $z\in \mathfrak O$ according to the branch of $H_n$ described above, and then both the left and the right hand sides are continued analytically along $\gamma$.
In this way this formula may be extended from progressive paths to rectangles in the $\zeta=e^{-H_n(z)}$ plane.

Constant $C=C_{n,1}$ in \eqref{2} depends on the normalization of $R_n$.
\end{remark}

Formula \eqref{2} is not totally satisfactory, since it has a number of undetermined parameters. Our next task is to clarify their behavior.
 
Using Assumption \ref{assumption1} and compactness argument we can choose a subsequence $\Lambda=\{n_k\}\subset \N$ such that
$$
h_n \to h, \quad D_n \to D, \quad \text{as } n\in \Lambda,
$$
so that by \eqref{2},
$$
\frac{1}{n}\, \log|R_n(z)| \capto -\Re \int_{a_1}^z \sqrt{\frac{D(t)}{A h(t)}}\, dt, \quad n\in \Lambda,
$$
for $z\in \mathfrak O$. Let $\Gamma$ be the Stahl's compact associated with $f$, i.e.~$f$ is holomorphic in $\C\setminus \Gamma$, and $\Gamma$ has the minimal capacity in the class  $\mathfrak A(\overline\C\sm \mathcal A)$.  Then, it follows from Theorem \ref{thm:stahl} that  $D/(Ah)=V/A$, with $V$ defining $\Gamma$. Since $\Lambda$ was an arbitrary convergent subsequence, we conclude that actually
$$
\lim_n H_n'=\lim_n \frac{D_n}{A h_n}=\frac{V}{A}=G'(\cdot, \infty).
$$
This establishes
\begin{lemma}\label{lemma1part1}
For the polynomials $D_n$ in \eqref{1} we have the representation $D_n(z)=V_n(z)\widetilde h_n(z) $, such that 
$$
V_n(z)=\prod_{k=1}^{p-2} (z-v_{k,n}) \to V(z), \quad n\to \infty,
$$ and $\widetilde h_n(z)=z^{p-2}+\dots$ satisfies $h_n - \widetilde h_n \to 0$ as $n\to \infty$.
\end{lemma}
Observe also that from our Conjecture \ref{assumption1} it follows that all zeros of $V$ are simple.
 
Since $H_n$ appears multiplied by $n$ in \eqref{2}, we need to estimate the rate of convergence of $H_n'$ to $G'$. Together with the Green function $G$ it is convenient to consider also 
$$
G_n(z,\infty):=\int_{a_1}^z \sqrt{\frac{V_n(t)}{A(t)}}\, dt.
$$
The following lemma is an elementary observation:
\begin{lemma}\label{lemma2}
We have 
$$
H_n'(z)=G'_n(z,\infty)\left( 1+\delta_{h,n}(z) + \varepsilon_1(z) \right)=G'(z,\infty)\left( 1+\delta_{h,n}(z) +\delta_{V,n}(z) + \varepsilon_2(z) \right),
$$
where
\begin{align}
\label{4}
\delta_{h,n}(z) := & \frac{\widetilde h_n(z)- h_n(z)}{2 h_n(z)}=\sum_{k=1}^{p-2} \frac{\beta_{k,n}}{z-z_{k,n}}, \quad \text{with } \quad  \beta_{k,n}:=\frac{\widetilde h_n(z_{k,n})}{2  h_n'(z_{k,n})}, \\ 
\label{5}
\delta_{V,n}(z) := & \sum_{k=1}^{p-2} \frac{\Delta v_{k,n}}{z-v_{k}}, \quad \text{with } \quad   \Delta v_{k,n}:= v_{k,n} -  v_{k}, \\
\nonumber
\varepsilon_j= &\, \mathcal O (\delta_{h,n}^2(z)) + \mathcal O (\delta_{V,n}^2(z)), \quad n\to \infty.
\end{align}
\end{lemma}
 \begin{proof}
 Applying the identity $\sqrt{1+\xi}=1+\xi/2+\mathcal O(\xi^2)$,  $\xi \to \infty$, we get
 $$
 H_n'(z)=\sqrt{\frac{V_n(z)\widetilde h_n(z) }{A(z) h_n(z)}}=G_n'(z,\infty)\left( 1+ \delta_{h,n}(z)  + \mathcal O (\delta_{h,n}^2(z))\right).
 $$
By Assumption \ref{assumption1}, zeros of $h_n$ are all simple for $n$ large enough. Hence, using the partial fraction decomposition for $(\widetilde h_n- h_n)/h_n$  we obtain the second identity in \eqref{4}.
 
 Finally, differentiating $G_n'(\cdot, \infty)$ with respect to its parameters $v_{k,n}$, we obtain that $G_n'(\cdot, \infty)=G'(\cdot, \infty) \left( 1+ \delta_{V,n}(z)  + \mathcal O (\delta_{V,n}^2(z))\right)$, with $\delta_{V,n}$ given in \eqref{5}. 
 \end{proof}
 
 In order to find the asymptotics for $\delta_h=\delta_{h,n}$ we need the following result:
 \begin{lemma}\label{lemma3}
 At any zero $z_{k,n}$ of $h_n$ we have
 \begin{equation}
\label{identAtzeros}
 n(n+1) D_n^2 = A h_n' D_n'-(A h_n''+ B h_n') D_n\,.
\end{equation}
 \end{lemma}
 \begin{proof}
 Differentiating \eqref{1} and evaluating the result at $z=z_{k,n}$  we get
 \begin{equation}\label{7}
-\left(A h_n'' + B h_n'+N^2 D_n \right) R_n' -N^2 D_n' R_n  =0,
\end{equation}
where
\begin{equation}\label{defN}
N:=\sqrt{n(n+1)}=n+\frac{1}{2}+\mathcal O\left( \frac{1}{n}\right), \quad n\to \infty.
\end{equation}
Also from \eqref{1}, for $z=z_{k,n}$,
 \begin{equation}\label{8}
  A h_n'   R_n' +N^2 D_n R_n =0.
\end{equation}
 \eqref{7}--\eqref{8} give us a homogeneous linear system on $(R_n(z), R_n'(z))$ with a non-trivial solution, since by the uniqueness theorem, at a regular point of \eqref{1}, both $R_n$ and $R_n'$ cannot vanish simultaneously. Hence, the determinant of this system is zero, which yields the assertion.
 \end{proof}

Using that $D_n= V_n\widetilde h_n $ we obtain from \eqref{identAtzeros},
$$
N^2 V_n^2 \widetilde h_n^2 = A \widetilde h_n' h_n' V_n+ \widetilde h_n (A V_n' h_n' -A h_n'' + B h_n') \quad \text{for } z=z_{k,n},
$$
and since by Assumption \ref{assumption1}, $h_n'(z_{k,n})\neq 0$,
\begin{equation}\label{9}
N^2\,\left( \frac{\widetilde h_n}{h_n'}\right)^2 = \frac{A}{V_n} \left(1+\frac{\widetilde h_n'- h_n'}{ h_n'} + \widetilde h_n \frac{A h_n' V_n' - A h_n'' + B h_n'}{(h_n')^2 A V-n}\right) \quad \text{for } z=z_{k,n}.
\end{equation}
As a consequence, we get the  following lemma:
\begin{lemma}
\label{lemma4}
For $\beta_{k,n}$ defined in \eqref{4},  
$$
\beta_{k,n}^2 =\frac{1}{4 N^2}\,  \frac{A(z_{k,n})}{V_n(z_{k,n})}  \, \left( 1+\mathcal O(\delta_{h,n})\right)=\frac{1}{4 N^2}\,  \frac{A(z_{k,n})}{V(z_{k,n})} \, \left( 1+\mathcal O(\delta_{h,n}) +\mathcal O(\delta_{V,n}) \right).
$$
\end{lemma}
Observe that the last identity is obtained applying also Lemma \ref{lemma2}.

Next, we use the possibility of the analytic continuation in \eqref{2} in order to derive the asymptotic identities on the unknown parameters.
\begin{lemma}
\label{lemma5}
Let $\gamma$ be a cycle (simple closed curve) in $\C\sm e_n$ enclosing two points, say $a_1, a_2 \in \mathcal A$, in such a way that the rest of points from $e_n$ are exterior to $\gamma$. Then
\begin{equation}
\label{sine}
N\, \oint _\gamma H_n'(t)\, dt=T(\gamma, f)(1  +\mathcal O(1/n)),
\end{equation}
with
$$
T(\gamma, f):=\pm \log \frac{\sin \pi \alpha_1 }{\sin \pi \alpha_2}+ 2\pi i m, \quad m\in \Z,
$$
where  the sign is uniquely determined by the branch of the square root and the orientation of the contour $\gamma$ chosen.
\end{lemma}
\begin{proof}
By Theorem \ref{theorem2}, for any $a_j\in \mathcal A$ there exists a progressive path $\gamma_j$ from $\mathfrak O$ to $\mathfrak O$ that is a closed Jordan curve separating $a_j$ from other points of $e_n$; assume $\gamma_j$ positively oriented with respect to $a_j$. 

Observe that both analytic germs $f$ and $R_{n}=Q_n f-P_n$ in $\mathfrak O$ allow for the analytic continuations along any such a path. Denote by $f_{\gamma_j}$ and $R_{n,\gamma_j}=Q_n f_{\gamma_j}-P_n$ the values of these functions that we obtain after the analytic continuations of $f$ and $R_n$, respectively, along $\gamma_j$. If we denote by $-\gamma_j$ the negatively oriented contour $\gamma_j$, then for $z\in \gamma_j\cap \mathfrak O$, 
$$
f_{\pm \gamma_j}(z)=f (z)e^{\pm 2\pi i \alpha_j}.
$$
Consider for instance paths $\gamma_1$ and $-\gamma_2$. By the definition of the residue,  $R_n=Q_n f -P_n$, we have 
\begin{equation}
\label{10}
\frac{R_{n,-\gamma_2}-R_{n}}{R_{n,\gamma_1}-R_{n}}(z)=\frac{f_{-\gamma_2}-f}{f_{\gamma_1}-f}(z).
\end{equation}
But
\begin{equation}
\label{11}
\frac{f_{-\gamma_2}-f}{f_{\gamma_1}-f}(z)=\frac{e^{-2\pi i \alpha_2}-1}{e^{2\pi i \alpha_1}-1}= - \frac{e^{-\pi i \alpha_2}}{e^{\pi i \alpha_1}}\, \frac{\sin \pi \alpha_2}{\sin \pi \alpha_1}=-\sqrt{\frac{f_{-\gamma_2}}{f_{\gamma_1}}} (z)\, \frac{\sin \pi \alpha_2}{\sin \pi \alpha_1}.
\end{equation}
On the other hand, during the analytic continuation the residue of the Pad\'e approximant picks up a dominant term; thus,  $R_n$ is geometrically small in $\mathfrak O$ in comparison with $R_{n,i}$, $i=1,2$ (see \eqref{2}), and from Theorem \ref{theorem2} we have for $z\in \mathfrak O$,
\begin{equation}
\label{12}
\begin{split}
\frac{R_{n,-\gamma_2}-R_{n}}{R_{n,\gamma_1}-R_{n}}(z)& =\frac{R_{n,2}}{R_{n,1}}(z)(1+o(1))\\ 
&= -\sqrt{\frac{f_{-\gamma_2}}{f_{\gamma_1}}} (z)\, \exp \left( (n+1/2) \oint_{\gamma_1-\gamma_2} H_n'(t) dt\right)\, (1+o(1)) \\
&= -\sqrt{\frac{f_{-\gamma_2}}{f_{\gamma_1}}} (z)\, \exp \left( N \oint_{\gamma_1-\gamma_2} H_n'(t) dt\right)\, (1+o(1))
\end{split}
\end{equation}
(observe that the negative sign comes from the fact that the term $A^{-1/4}$ is multiplied by $\pm i$ after its analytic continuation; orientations of $\gamma_1$ and $-\gamma_2$ are opposite, so after division we gain the $-1$ factor). Identities \eqref{10}--\eqref{12} yield the assertion with $\gamma=\gamma_1-\gamma_2$ or any any  cycle homotopic to it in $\C\sm e_n$. In order to extend the theorem to an arbitrary cycle $\gamma$  in $\C\sm e_n$, we observe that if during the homotopic deformation of the contour we cross a pair of adjacent zeros of $h_n$ and $\widetilde h_n$ (see Lemma \ref{lemma1part1}), both $R_{n,\gamma_1}$ and $R_{n,-\gamma_2}$ gain a change of sign, so that  \eqref{sine} remains valid.
\end{proof}

%
%
%

We introduce the Riemann surface $ \mathcal R $ defined by the equation $w^2=A(z) V(z)$. It is a hyperelliptic Riemann surface that can be considered as a two-sheeted covering of $\overline{\C}$, $\mathcal R=\{ \bm z =(z, w)\in \C^2 \}$, with two sheets, $ \mathcal R^{(1)}$ and $ \mathcal R^{(2)}$, cut along Stahl's compact $\Gamma$ and glued together in the standard way. From Assumption \ref{assumption1} and using the Riemann-Hurwitz formula we easily see that the genus of $\mathcal R$ is $p-2$. The canonical projection $ \pi:\,  \mathcal R\to \overline{\C}$ is given by $\pi(\bm z)=z$ for $ \bm z =(z, w)\in \C^2 $. We denote $\bm z^{(j)} =\pi^{-1}(z) \cap \mathcal R^{(j)}$, $j=1, 2$, and we convene that sheet $\mathcal R^{(1)}$ over $\C\setminus \Gamma$ is specified by the condition  $w/z^2 \to 1$ as $\bm z  \to \bm \infty^{(1)}\in \mathcal R^{(1)}$. 
In this way, function $w=(AV)^{1/2}$ is single-valued on $\mathcal R$. 
Unless specified otherwise, we identify the first sheet $\mathcal R^{(1)}$ with the domain $\overline{\C}\setminus \Gamma=\pi(\mathcal R^{(1)})$. We construct analogously the Riemann surface $ \mathcal R_n $ defined by the equation $w^2=A(z) V_n(z)$. Again, by Assumption \ref{assumption1}, the genus of $\mathcal R_n$ is $p-2$ for $n$ large enough.

Note that a homology basis of cycles of $\mathcal R$ and $\mathcal R_n$ can be constructed from an integer combination of cycles $\gamma_{ij}=\gamma(a_i, a_j)=\gamma_i - \gamma_j$ considered in the proof of Lemma \ref{lemma5}. Thus, \eqref{sine} is valid for any cycle $\gamma$ on $\mathcal R$ with the right hand side $T(\gamma, f)$ depending on the representation of $\gamma$ in terms of the basis of cycles $\gamma_{ij}$. 
We select a homology basis of cycles $\Gamma_j$ of $\mathcal R$, $j=1, 2,\dots, 2p-4$, in such a way that, in the standard terminology, $\Gamma_j$ are the $\mathfrak a$-cycles of $\mathcal R$ when $j=1, \dots, p-2$, while for $j=p-1, \dots, 2p-4$ they form the $\mathfrak b$-cycles. 

Next, we introduce a notation for some special functions and meromorphic differentials on $\mathcal R$. Function
$$
\ell(z, \bm t):=\sqrt{\frac{V(t)}{A(t)}}\frac{1}{t-z}
$$
can be regarded as a meromorphic on $\mathcal R$ in both variables;
$$
d\omega_k (\bm t)=\ell(v_k,\bm t)\,dt , \quad k=1, \dots, p-2,
$$
is a basis of holomorphic differentials on $\mathcal R$, and  correspondingly,
$$
u_k(\bm z):=\int_{a_1}^{\bm z} d\omega_k, \quad k=1,\dots, p-2,
$$
form a basis of integrals of the first kind (these are multivalued and analytic functions on $\mathcal R$ having a constant increment along any cycle). 

Additionally, 
\begin{equation} \label{functionV}
\theta(\bm z,\bm \zeta):=\sqrt{\frac{A(z)}{V(z)}}\int_{a_1}^{\bm \zeta} \ell(z,\bm t)dt
\end{equation}	
can be also considered as an analytic function on $\mathcal R$ in both variables (multivalued  in $\bm \zeta$). Lemma \ref{lemma2} and Lemma \ref{lemma4} render that with an appropriate choice of $\bm z_{k,n}=\pi^{-1}(z_{k,n})$, that means, either $\bm z_{k,n}^{(1)}$ or $\bm z_{k,n}^{(2)}$, for $z\in \C\setminus \Gamma$, 
\begin{equation}
\label{14}
N   H_n(z ) =N  G(z,\infty)  +\sum_{k=1}^{p-2} d_{k,n}\, u_k(z)+\frac{1}{2}\, \sum_{k=1}^{p-2}  \theta( \bm z_{k,n} ,z)+ \mathcal O(1/n) ,
\end{equation}
where  $d_{k,n}:=N(v_{k,n}-v_k)= N \Delta v_{k,n}$, and as usual, we identify  $\C\setminus \Gamma$ with the first sheet of $\mathcal R$. 

Let us work out the system of equations on the unknown parameters. Given a closed contour (cycle) $\gamma$ on $\mathcal R$, we denote by
$$
\Theta(\bm z;\gamma):=\Delta_\gamma \theta(\bm z,\bm \zeta)\bigg|_{\bm \zeta\in \gamma}=\sqrt{\frac{A(z)}{V(z)}}\oint_\gamma \sqrt{\frac{V(t)}{A(t)}}\frac{dt}{t-z}, \quad \bm z\in \mathcal R \sm \gamma,
$$
the period of $\theta(\bm z, \cdot)$ along $\gamma$. 
A direct verification shows that $\Theta(\cdot;\gamma)$ can be analytically continued on $\mathcal R$ as an integral of the first kind, so that for suitably chosen $c_k(\gamma)\in \C$,
\begin{equation} \label{vToAbelian}
\Theta(\bm z;\gamma)=\sum_{k=1}^{p-2} c_k(\gamma) u_k(\bm z).
\end{equation}	
By \eqref{sine},
\begin{equation}
\label{13}
N\, \oint _{\Gamma_j} H_n'(t)\, dt= T(\Gamma_j, f)(1 +\mathcal O(1/n)) \mod{2\pi i}, \quad j=1, 2,\dots, 2p-4,
\end{equation}
and in view of \eqref{14},  equation \eqref{13} may be written as
\begin{equation}
\label{15}
N\, \oint _{\Gamma_j} G'(t,\infty)\, dt= T(\Gamma_j, f)-\sum_{k=1}^{p-2} d_{k,n}\, \oint_{\Gamma_j} d\omega_k-\frac{1}{2}\, \sum_{k=1}^{p-2}   \Theta(\bm z_{k,n} , \Gamma_j) +\mathcal O(1/n) \mod{2\pi i}, 
\end{equation}
with $ j=1, 2,\dots, 2p-4$. This is a system of $2p-4$ equations on $2p-4$ unknowns $d_{1,n}, \dots, d_{p-2,n}$, $\bm z_{1,n},  \dots, \bm z_{p-2,n}$, that  may be equivalently written in any basis $\Gamma_j$. From the general theory of Riemann surfaces it follows that matrix   
$$
\left(\oint_{\Gamma_j} d\omega_k\right)_{j,k=1}^{p-2}
$$
is invertible. 
Then, first $p-2$ equations in \eqref{15} may be explicitly solved for $d_{k,n}$. Substitution of those $d_{k,n}$'s in the remaining equations and the use of \eqref{vToAbelian} reduces the situation to the standard Jacobi inversion problem, which as it is well known, is uniquely solvable for any non-special divisor\footnote{In the situation when the divisor is special, we have $\deg h_n<p-2$ and for such an $n$ the normality is lost.}. 
Hence, system \eqref{15} is uniquely solvable for any right hand side. 

%
%

%
%
%

\begin{remark}
Under Assumption \ref{assumption1}, $\Delta v_{k,n}=\mathcal O(1/n)$ and 
all the remainders in \eqref{15} are  $\mathcal O(1/n)$, which is the accuracy for determining $d_{k,n}$ by these equations. 
\end{remark}

Now we can simplify the asymptotic formula \eqref{2}  from Theorem \ref{theorem2}. Since
$$
h_n^{3/4} D_n^{-1/4}=h_n^{1/2} V^{-1/4}(1+\mathcal O(1/n)),
$$
we get 
$$
R_n(z) =  C_{n,1}\, \frac{(f h_n)^{1/2}}{( AV)^{1/4}} (z)\, e^{-N H_n(z)} \, \left( 1+\mathcal O(1/n) \right),  
$$
and $H_n$ can be replaced by $\mathcal H_n$, the leading term in its asymptotic formula \eqref{14}:
\begin{equation} \label{Hcal}
\mathcal H_n(z)=   G(z,\infty)  +\frac{1}{N}\, \sum_{k=1}^{p-2} d_{k,n}\, u_k(z)+\frac{1}{2N}\, \sum_{k=1}^{p-2}   \theta( \bm z_{k,n},z).
\end{equation}	

Finally, using the analytic continuation of $R_n$ along a progressive path around an $a\in \mathcal A$ (if we take $a=a_1$, then $H_n$ just changes sign during this analytic continuation) and solving the system
$$
\begin{cases}
Q_n f-P_n &= R_n \\
Q_n f_1-P_n &= R_{n,1} \\
\end{cases}
$$
for $Q_n$ we obtain $Q_n=(R_{n,1}-R_n)/(f_1-f)$, from where the exterior asymptotics has the form
$$
Q_n(z)= C_{n,1} \, \frac{ h_n^{1/2}}{ f ^{1/2} (AV)^{1/4}} (z)\, e^{N H_n(z)} \, \left( 1+\mathcal O(z) \right).
$$
We summarize our findings in the following theorem:
\begin{theorem}\label{thm:wkb}
For a normal index $n\in \N$ let $N=\sqrt{n(n+1)}$ and $\mathcal H_n$ be as defined in \eqref{Hcal}, 
with the coefficients $d_{k,n}$ and $\bm z_{k,n}$ determined by equations  \eqref{15}. Then,  with an appropriate normalization,
$$
R_n(z) =  C_{n,1}  \frac{(f h_n)^{1/2}}{( AV)^{1/4}} (z)\, e^{-N \mathcal H_n(z)} \, \left( 1+\mathcal O(1/n) \right),  
$$
and 
\begin{equation} \label{WKBforQn}
Q_n(z)= C_{n,2} \, \frac{ h_n^{1/2}}{ f ^{1/2} (AV)^{1/4}} (z)\, e^{N \mathcal H_n(z)} \, \left( 1+\mathcal O(1/n) \right),
\end{equation}	
for $z$ on a compact subsets of $\C\sm \Gamma$.
\end{theorem}

\section{Case of $p=3$: Chebotarev compact and the Riemann surface} \label{sec:Chebotarev}

In the rest of the paper we concentrate on the particular case studied in \cite{MR891770}, when $p=3$, $a_1$, $a_2$ and $a_3$ are 3 non-collinear points on the complex plane $\C$, and $   \Gamma$ is the Chebotarev compact, i.e.~the set of minimal capacity containing these points. 
Recall (see Theorem \ref{thm:stahl}) that there exists a point $v $ in the convex hull of $\mathcal A$, called the \emph{center} of the Chebotarev compact, such that with
\begin{equation*}
\label{R}
   A(z) =(z-a_1) (z-a_2)  (z-a_3), \quad {   V(z)}=z-v, \quad \text{and} \quad {  T(z)}=\left( \frac{V(z)}{A(z)}\right)^{1/2},
\end{equation*}
where the branch of $T$ in $\C\setminus \Gamma$ is specified by $\lim_{z\to \infty} z\, T(z)=1$, it is determined uniquely by
the set of equations
$$
  \Re \int_{a_1}^{v} T(t)\, dt =\Re \int_{a_2}^{v} T(t)\, dt =0  .
$$
Furthermore, 
$$
\Gamma= \Gamma_1 \cup \Gamma_2 \cup \Gamma_3, 
$$
with
$$
\Gamma_j := \left \{ z\in \C:\, \Re \int_{a_j}^z T(t)\, dt =0 \right\},
$$ 
the arc of $\Gamma$ joining $a_j$ with $v$, $j=1, 2, 3$. 
We introduce also the orthogonal trajectories 
$$
{  \Gamma^\perp}  := \left \{ z\in \C:\, \Im \int_{v}^z T(t)\, dt   =0 \right\}, 
$$
which consist of 3 unbounded rays emanating from $v$, 
as well as
$$
{   \gamma_j^\perp}  := \left \{ z\in \C:\, \Im \int_{a_j}^z T(t)\, dt   =0 \right\}, \quad j=1,2,3;
$$
 each $ \gamma_j^\perp$ is an unbounded ray emanating from $a_j$, see Fig.~\ref{fig:Gamma}.
\begin{figure}[h]
\centering \begin{overpic}[scale=1.1]%
{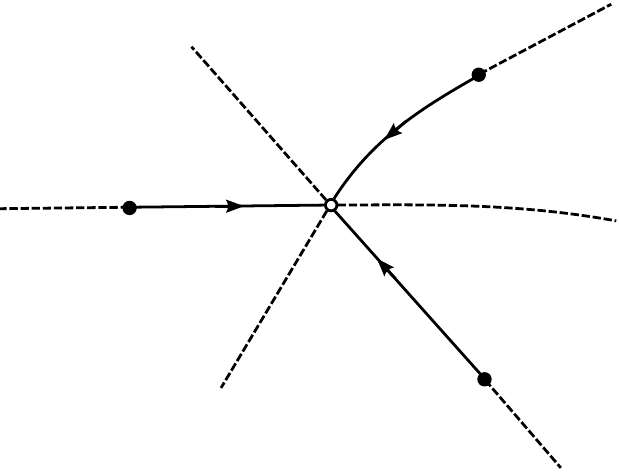}%
\put(52.3,37.5){$v$}
\put(19,38){$a_1$}
\put(75,10){$a_2$}
\put(74,67){$a_3$}
\put(32,38){\small $\Gamma_1$}
\put(75,38){\small $\Gamma_1^\perp$}
\put(65,20){\small $\Gamma_2$}
\put(29,60){\small $\Gamma_2^\perp$}
\put(65,51){\small $\Gamma_3$}
\put(29,10){\small $\Gamma_3^\perp$}
\put(5,45){\small $\gamma_1^\perp$}
\put(88,5){\small $\gamma_2^\perp$}
\put(88,67){\small $\gamma_3^\perp$}
\end{overpic}
\caption{$\Gamma$ and $\Gamma^\perp$.}
\label{fig:Gamma}
\end{figure}
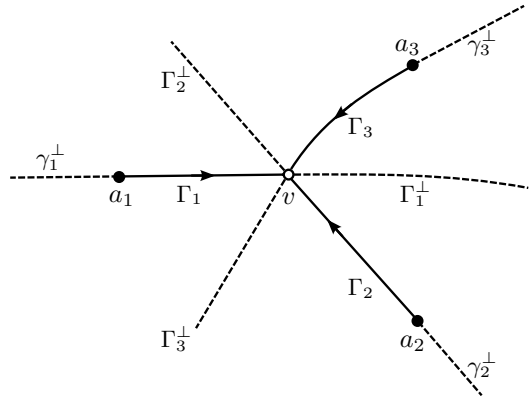

Contour $\gamma_1^\perp \cup \Gamma  \cup \gamma_2^\perp$ splits $\C\setminus \Gamma$ into two simply connected domains. We denote by ${D_+}$ the domain containing $a_3$ on its boundary, and ${D_-}$ the complementary one.

On the three subarcs of $\Gamma$ we fix the orientation ``from $a_j$ to $v$'', while on the arcs of $\Gamma^\perp$ we choose the orientation ``from $v$ to infinity''. This induces the left (``$+$'') and right (``$-$'') sides and boundary values.  

The \emph{equilibrium measure} $  \lambda=\lambda_\Gamma$ on $\Gamma$ has the form
$$
d\lambda(z) = \frac{1}{\pi i}\, T_-(z) dz
$$
(compare with \eqref{equilibriumMeasureGeneral}). We denote also
$$
{  m_j} = \lambda(\Gamma_j)=\frac{1}{\pi i}\, \int_{a_j}^{v}T_-(t) dt,  \quad j=1,2,3,
$$
so that $m_1+m_2+m_3=1$. 

Define in $\C\setminus \Gamma$
\begin{equation}
\label{defPhiconformal}
{  \Phi(z)}=  \exp\left( \int_{v} ^z T(t)\, dt\right),
\end{equation}
normalized by the condition
$$
\lim_{\stackrel{z\to v}{ z\in \Gamma_1^\perp}} \Phi(z)=1;
$$
observe that $\Phi$ coincides up to a multiplicative constant with $  \exp\left( G(\cdot ,\infty)\right)$ introduced in  \eqref{defGreenF}. It is a conformal mapping of the exterior $\C\setminus \Gamma$ onto the exterior of the unit circle, such that
\begin{equation}
\label{def:Phi}
\Phi(z)={   c}\, z + \mathcal O(1), \quad z\to \infty,
\end{equation}
with $1/c$ coinciding, again up to a factor of absolute value $1$, with the logarithmic capacity of $\Gamma$. 
Direct calculation allows to establish the following lemma:
\begin{lemma} \label{lemma1}
For $z\in { \Gamma^\circ }:=\Gamma\setminus \{v, a_1, a_2, a_3 \}$ and with the orientation shown on Figure~\ref{fig:Gamma},
\begin{equation}
\label{bdryconditionsPhi}
\Phi_-(z)\Phi_+(z) =\kappa_j,\quad  z \in \Gamma_j^\circ:=\Gamma_j\setminus \{v, a_j \}, 
\end{equation}
with
\begin{equation}
\label{def:kappa2}
{  \kappa_1 }=e^{2\pi i (m_3-m_2)}, \quad {  \kappa_2 }=e^{-2\pi i  m_2 }, \quad  {  \kappa_3 }=e^{2\pi i m_3}, 
\end{equation}
so that $|\kappa_j|=1$ and $\kappa_2 \kappa_3=\kappa_1$.
\end{lemma}

\begin{figure}[h]
\centering \begin{overpic}[scale=0.7]%
{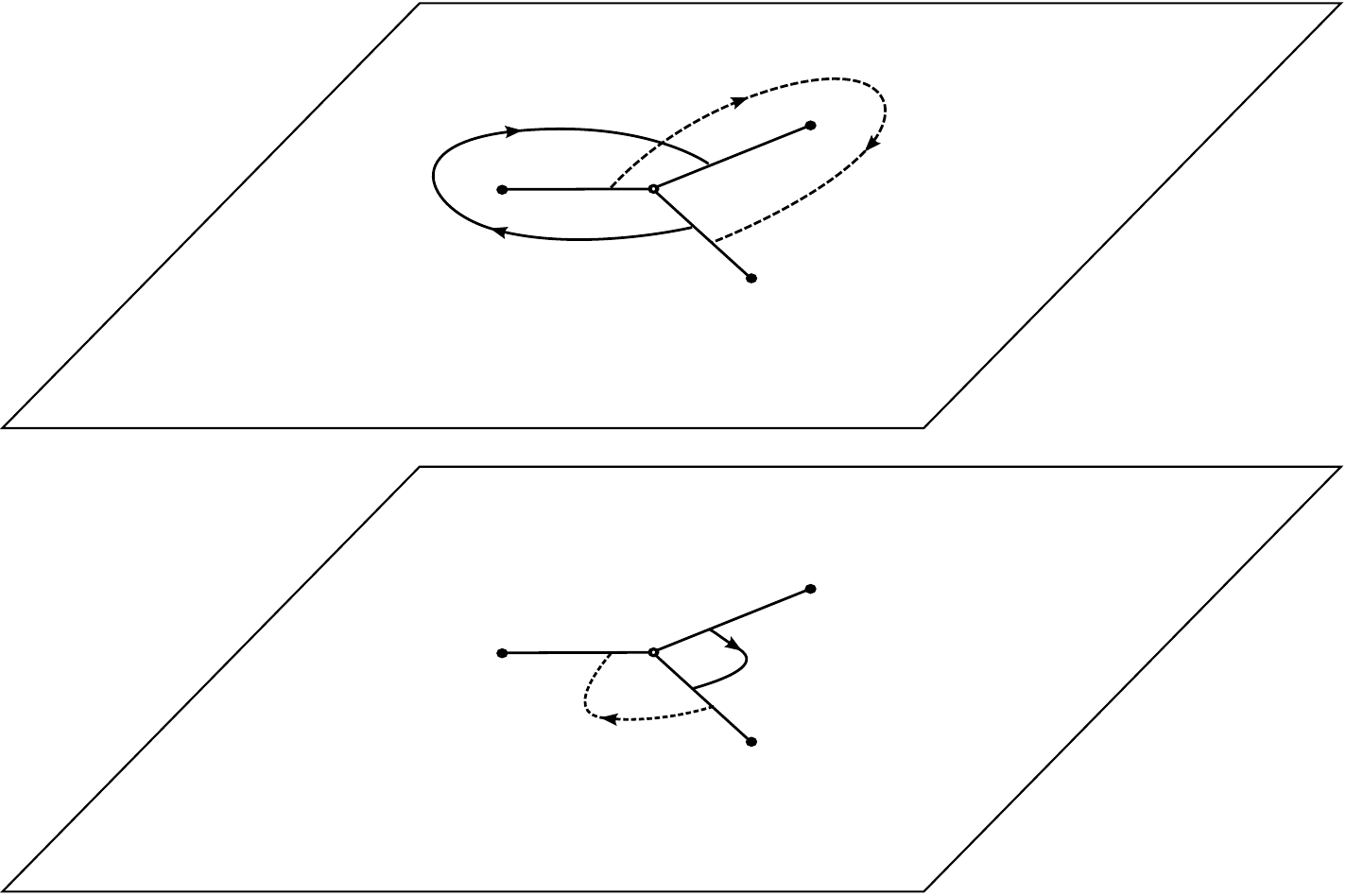}%
\put(46.5,50.5){\small $v$}
\put(46.5,16){\small $v$}
\put(33.5,52){\small $a_1$}
\put(33.5,17){\small $a_1$}
\put(61,57){\small $a_3$}
\put(61,22){\small $a_3$}
\put(56,43.5){\small $a_2$}
\put(56,9){\small $a_2$}
\put(10,40){\small $\mathcal R^{(1)}$}
\put(10,5){\small $\mathcal R^{(2)}$}
\put(30,45){\small $\mathfrak a$-cycle}
\put(67.5,54){\small $\mathfrak b$-cycle}
\end{overpic}
\caption{Cycles on $\mathcal R$.}
\label{fig:cycles}
\end{figure}
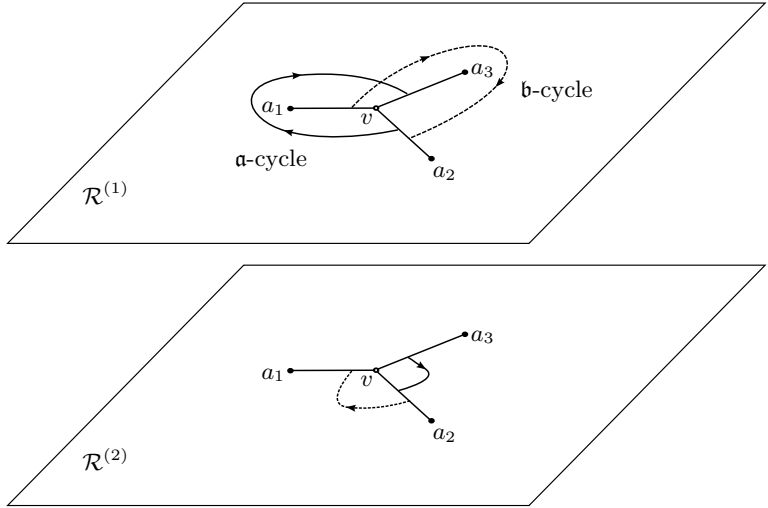

As before, we consider the Riemann surface $ {\mathcal R} $ defined by the equation $w^2=A(z) V(z)$. Now it is an elliptic Riemann surface that can be considered as a two-sheeted covering of $\overline{\C}$, $\mathcal R=\{ \bm z =(z, w)\in \C^2 \}$, with two sheets, $ \mathcal R^{(1)}$ and $ \mathcal R^{(2)}$, cut along $\Gamma$ and glued together in the standard way. The canonical projection $ {\pi}:\,  \mathcal R\to \overline{\C}$ is given by $\pi(\bm z)=z$ for $ \bm z =(z, w)\in \C^2 $. As in Section \ref{sec:wkb}, we denote $\bm z^{(j)} =\pi^{-1}(z) \cap \mathcal R^{(j)}$, $j=1, 2$, and we convene that sheet $\mathcal R^{(1)}$ over $\C\setminus \Gamma$ is specified by the condition  $w/z^2 \to 1$ as $\bm z  \to \bm \infty^{(1)}\in \mathcal R^{(1)}$. 
In this way, function $w=(AV)^{1/2}$ is single-valued on $\mathcal R$, with 
$$
w(\bm z)=\begin{cases}
T(z)/V(z), & \text{for } \bm z= \bm z^{(1)}\in \mathcal R^{(1)} ,\\
-T(z)/V(z), & \text{for } \bm z= \bm z^{(2)}\in \mathcal R^{(2)} .
\end{cases}
$$
Again, we identify the first sheet $\mathcal R^{(1)}$ with the domain $\overline{\C}\setminus \Gamma=\pi(\mathcal R^{(1)})$.
We also denote $\mathcal D_\pm^{(j)}=\pi^{-1}(D_\pm)\cap \mathcal R^{(j)}$, $j=1,2$.

We define the canonical homology basis of cycles as in Figure~\ref{fig:cycles}: the $\mathfrak a$-cycle encloses $v$ and $a_1$, while the $\mathfrak b$-cycle goes around $v$ and $a_3$. Both are oriented as indicated in the figure, so that at their unique intersection point on $\mathcal R$ their tangent vectors form a right pair. 

The normal form of the Riemann surface $\mathcal R$ is the polygon (rectangle) $\widetilde{\mathcal R}$ with sides  $\mathfrak a \mathfrak b \mathfrak a^{-1} \mathfrak b^{-1}$ (see Figure \ref{fig:rectangle})\footnote{One of the authors of \cite{Baratchart:2011fk} kindly pointed out to us that a similar figure is contained in the cited paper.}. 

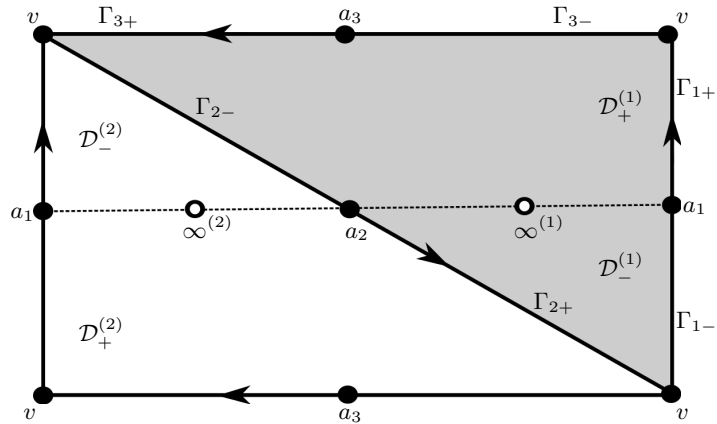
\begin{figure}[h]
\centering \begin{overpic}[scale=0.9]%
{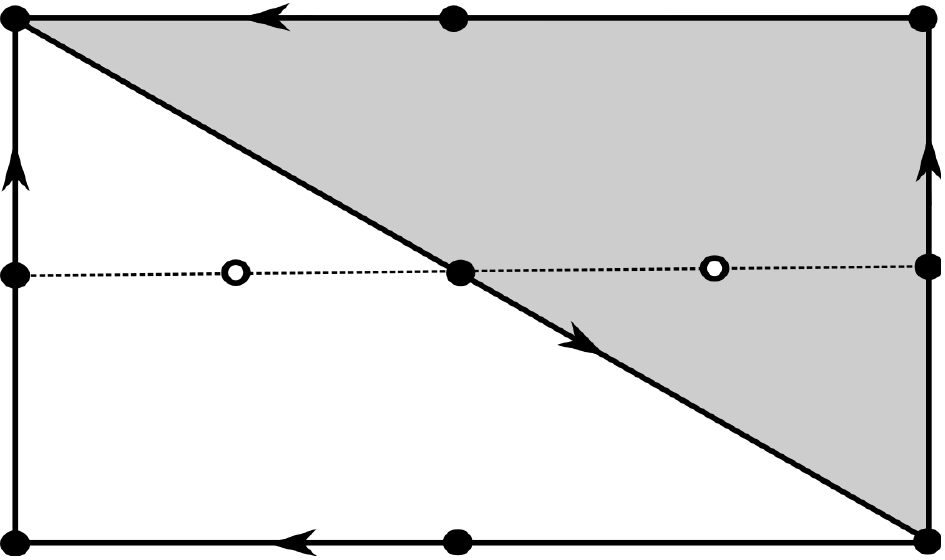}%
\put(-1.5,-2){\small $v$}
\put(-1,59){\small $v$}
\put(99,-2){\small $v$}
\put(99,59){\small $v$}
\put(-3.5,29){\small $a_1$}
\put(100,30){\small $a_1$}
\put(47,-2){\small $a_3$}
\put(47,59.5){\small $a_3$}
\put(48,26){\small $a_2$}
\put(23,26){\small $\infty^{(2)}$}
\put(74,26){\small $\infty^{(1)}$}
\put(7,40){\small $\mathcal D_-^{(2)}$}
\put(7,10){\small $\mathcal D_+^{(2)}$}
\put(87,20){\small $\mathcal D_-^{(1)}$}
\put(87,45){\small $\mathcal D_+^{(1)}$}
\put(99,48){\small $\Gamma_{1+}$}
\put(99,12){\small $\Gamma_{1-}$}
\put(10,59){\small $\Gamma_{3+}$}
\put(80,59){\small $\Gamma_{3-}$}
\put(77,15){\small $\Gamma_{2+}$}
\put(25,45){\small $\Gamma_{2-}$}
\end{overpic}
\caption{Polygon $\widetilde{\mathcal R}$. Shaded domain corresponds to the first sheet $\mathcal R^{(1)}$, and the dotted line corresponds to $\pi^{-1}(\gamma_1^\perp \cup \gamma_2^\perp)$.}
\label{fig:rectangle}
\end{figure}

We introduce also some notation, slightly different from that used in Section \ref{sec:wkb}, related to differentials on $\mathcal R$ and their periods. For integer $k$, denote
\begin{equation}
\label{def:holomorphicdiff}
d\nu_k (\bm z)= \frac{z ^k d  z}{w(\bm z)}.
\end{equation}
Then $d\nu_0$ is, up to a constant multiple, the only holomorphic differential (abelian differential of the first kind) on $\mathcal R$. 

For two points $\bm r_1, \bm r_2\in \mathcal R$ we denote by $ {\Omega_{\bm r_1,\bm r_2}}$ the normalized differential of the third kind, such that it has  only simple poles at $\bm r_1$, with residue $+1$, and at $\bm r_2$, with residue $-1$, and with a vanishing $\mathfrak b$-period. In particular, if $ {\bm z^*}=( z^*,  w^*)\in \mathcal R$,
\begin{equation}
\label{def:Omega}
\begin{split}
d \Omega_{\bm z^*,\bm \infty^{(1)}}& =\frac{1}{2 w(\bm z)} \left(\frac{w(\bm z)+w^*}{z-z^*} + z + \delta \right) dz, \\
\delta & = - \left(\oint_{\mathfrak b} \frac{1}{w(\bm z)} \left(\frac{w(\bm z)+w^*}{z-z^*} + z  \right) dz \right) \bigg / \left(\oint_{\mathfrak b} \frac{dz}{w(\bm z)}   \right).
\end{split}
\end{equation}

With the orientation of $\Gamma_j$ specified in Figure \ref{fig:Gamma} we define 
\begin{equation}
\label{defMs}
{M_j^k}:=-\frac{1}{2\pi i}\,\int_{\Gamma_j} \frac{t^k dt}{w_+(t)  }, \quad k\in \N\cup \{0\},  \quad j=1,2, 3,
\end{equation}
so that
$$
\oint_{\mathfrak a} d\nu_k =-4\pi i\, M_1^k , \quad \oint_{\mathfrak b} d\nu_k=-4\pi i\, M_3^k , \quad k\in \N\cup \{0\}.
$$
In particular,  $M_j^0\neq 0$, $j=1,2, 3$, and $\Im (\tau)<0$, with
\begin{equation}
\label{defTau}
  {\tau}:= \frac{M_1^0}{M_3^0}.
\end{equation}
Direct computation using the Cauchy integral formula shows that
\begin{equation*}
\label{identitiesM}
M_1^0 + M_2^0 + M_3^0  =0, \quad M_1^1 + M_2^1 + M_3^1 =\frac{1}{2}, \quad
M_1^2 + M_2^2 + M_3^2  =\frac{\mathcal S}{2}, 
\end{equation*}
where we use the notation
\begin{equation}
\label{def:average}
 {\mathcal S} :=\frac{v+a_1+a_2+a_3}{2}.
\end{equation}

We reserve the notation $d\nu_0^*$ for the \emph{normalized} differential of the first kind, whose $\mathfrak b$-period is equal $2\pi i$:
\begin{equation} \label{normalized}
 {d\nu_0^*(z)}=-\frac{1}{2 M_3^0 }\frac{dz}{w(\bm z)}=-\frac{1}{2 M_3^0 }\, d \nu_0. 
\end{equation}	
Observe that the $\mathfrak a$-period of $\nu_0^*$ is  $2\pi i \tau$.

\section{Riemann and Hilbert, or the non-linear steepest descent analysis} \label{sec:RH}

Now we are ready to return to the Pad\'e approximants of the function 
$$
 {f(z)}=(z-a_1)^{\alpha_1} (z-a_2)^{\alpha_2} (z-a_3)^{\alpha_3},
$$
with $\alpha_1, \alpha_2, \alpha_3 \in \R\setminus \Z$, such that $\alpha_1+\alpha_2+\alpha_3=0$. For the sake of simplicity of the analysis we assume additionally that  $\alpha_j>-1$, $j=1,2,3$. As before, we specify the branch in $\C \setminus \Gamma$ by $f(\infty)=1$ and  agree in denoting by $f^{1/2}$ the branch of the square root in $\C\setminus \Gamma$ given by $f^{1/2}(\infty)=1$.
 
We need to introduce an additional piece of notation: for $j=1,2,3$, let
\begin{equation}
\label{def:constants}
{ \tau_j}= e^{-i \pi \alpha_j}, \quad {  t_j}=  2 i \sin (\pi \alpha_j),\quad {  s_{n,j}}= t_j \kappa_j^n,
\end{equation}
with $\kappa_j$ defined in \eqref{def:kappa2}. Observe that $\tau_1 \tau_2 \tau_3 =1$, and 
\begin{equation}
\label{identitiies_tau}
\tau_{j+1} t_{j-1} + t_j +t_{j+1} \tau_{j-1}^{-1}=   \tau_{j-1} t_{j+1} + t_j +t_{j-1} \tau_{j+1}^{-1}=0, \quad j=1, 2, 3,
\end{equation}
where the subindices are taken mod 3.

Let us collapse the contour of integration in \eqref{orthogonality1} onto $\Gamma$; as a consequence, function $f$ induces on $\Gamma$ the weight $\rho$,
$$
{\rho(z)}=f_-(z)-f_+(z)=\frac{t_j}{\tau_j}\, f_+(z)=(\tau_j^{-2}-1) f_+(z), \quad z\in \Gamma_j\setminus \{v, a_j \}, \quad j=1, 2, 3, 
$$
so that the orthogonality condition \eqref{orthogonality1} can be rewritten as 
\begin{equation}
\label{orthogonality2}
\int_\Gamma  t^k Q_n(t) \rho(t)\, dt =0, \quad k=0, 1, \dots, n-1.
\end{equation}
By our assumption that $\alpha_j>-1$, the weight is integrable on $\Gamma$, and the regularity of $f$ at infinity implies that
$$
\lim_{\stackrel{z\to v}{z\in \Gamma_1}} \rho(z) + \lim_{\stackrel{z\to v}{z\in \Gamma_2}} \rho(z) + \lim_{\stackrel{z\to v}{z\in \Gamma_3}} \rho(z) =0.
$$
Standard arguments show that there is an integral formula for the residue $R_n$:
$$
R_n(z)= \frac{1}{2\pi i} \int_{\Gamma}   \frac{Q_n(t) \rho(t)}{t-z}dt , \quad z \in \overline{\C}\setminus \Gamma.
$$

This allows us to formulate the Riemann-Hilbert problem for $Q_n$ and $R_n$.  
Let $\sigma_3$ denote the third Pauli matrix,
$$
 {\sigma_3}=\begin{pmatrix}
1 & 0 \\
0 & -1
\end{pmatrix},
$$
and for any scalar $a$ we use the notation
$$
a^{\sigma_3}=\begin{pmatrix}
a & 0 \\
0 & a^{-1}
\end{pmatrix}.
$$ 
We seek the matrix-valued and analytic function $\bm Y=\bm Y(\cdot; n): \C\setminus \Gamma \to \C^{2\times 2}$, such that:
\begin{enumerate}
\item[(RH-Y1)] It has continuous boundary values $\bm Y_\pm$ on both sides of $\Gamma^\circ$, and with the specified orientation of $\Gamma$,
$$
\bm Y_+(z)=\bm Y_-(z)\, \begin{pmatrix}
1 & \rho(z) \\
0 & 1
\end{pmatrix}, \quad z\in \Gamma^\circ. 
$$
\item[(RH-Y2)] $\bm Y(z)= (\bm I +\mathcal O(1/z))\, z^{n\sigma_3}$, as $z\to \infty$.
\item[(RH-Y3)] As $z\to a_j$, $z \in \mathbb C \setminus \Gamma$, $j=1, 2, 3$,
 \begin{equation*}\label{RHPY3}
            \bm Y(z)=\left\{
            \begin{array}{cl}
               \mathcal O\begin{pmatrix}
                    1 & |z-a_j|^{\alpha_j} \\
                    1 & |z-a_j|^{\alpha_j}
                \end{pmatrix},
                &\mbox{if $\alpha_j<0$,} \\[2ex]
                \mathcal  O\begin{pmatrix}
                    1 & 1 \\
                    1 & 1
                \end{pmatrix},
                &\mbox{if $\alpha_j>0$.}
            \end{array}\right.
        \end{equation*}
        \item[(RH-Y4)] As $z\to v$, $z \in \mathbb C \setminus \Gamma$, 
 \begin{equation*}\label{RHPY4}
            \bm Y(z)=  \mathcal O\begin{pmatrix}
                    1 & \log |z-v| \\
                    1 & \log |z-v| 
                \end{pmatrix}.         
                \end{equation*}
\end{enumerate}
From the fundamental work of Fokas, Its and Kitaev \cite{Fokas92} it follows that
\begin{theorem}
    The matrix valued function $\bm Y(z)$ given by
    \begin{equation*} \label{RHPYsolution}
      \bm  Y(z) =
        \begin{pmatrix}
           Q_n(z) & R_n(z) \\[2ex]
            -2\pi i \gamma_{n-1}^2 Q_{n-1}(z) & -2\pi i \gamma_{n-1}^2  R_{n-1}(z)
        \end{pmatrix}
    \end{equation*}
    is the unique solution of (RH-Y1)--(RH-Y4), where  $Q_n$ is the monic polynomial of degree $n$ satisfying \eqref{orthogonality2} and $ {\gamma_n}$ is the leading coefficient of
    the corresponding orthonormal polynomial.
\end{theorem}
This result is complemented with the non-linear steepest descent method of Deift and Zhou \cite{MR99g:34038, MR98b:35155, MR94d:35143, MR2000g:47048, MR98k:47097}: we need to perform a number of explicit and invertible transformations of (RH-Y1)--(RH-Y4) in order to reach a boundary value problem with jumps asymptotically close to the identity and a regular behavior at infinity.   Two of the main ingredients of this analysis are the outer (global) parametrix and the local model at the Chebotarev center, that we explain next.

\subsection{Global parametrix} \label{sec:outer}

For $n\in \N$ we need to find  an analytic matrix-valued function $\bm N_n=\bm N:\,  \C\setminus \Gamma \to \C^{2\times 2}$,  such that
\begin{enumerate}
\item[(RH-N1)] It has continuous boundary values $\bm N_\pm$ on both sides of $\Gamma^\circ$, and with the  orientation ``from $a_j$ to $v$'' of $\Gamma$,
\begin{equation}
\label{boundaryforN}
\bm N_+(z)=\bm N_-(z)\, \begin{pmatrix}
0 & s_{n,j} \\
-1/s_{n,j}& 0
\end{pmatrix}, \quad z\in \Gamma_j^\circ.
\end{equation}
\item[(RH-N2)] $\bm N(z)=  \bm I +\mathcal O(1/z))$, as $z\to \infty$.
\item[(RH-N3)] As $z\to a_j$, $z \in \mathbb C \setminus \Gamma$, $j=1, 2, 3$,
$$
\bm N(z)= \mathcal O(|z-a_j|^{-1/4}).
$$
 As $z\to v$, $z \in \mathbb C \setminus \Gamma$, 
 $$
\bm N(z)= \mathcal O(|z-v|^{-1/4}).
$$
\end{enumerate}
Constants $s_{n,j}$ were defined in \eqref{def:constants}; hence, the dependence on $n$ resides only in the boundary condition \eqref{boundaryforN}. 

On the compact subsets of $\C\setminus \Gamma$ this problem is asymptotically close to the boundary value problem for the following matrix,
\begin{equation}
\label{def:matrixT}
\bm T(z):=c^{n \sigma_3} \bm Y(z) \Phi^{-n \sigma_3}(z) f^{\sigma_3/2}(z),
\end{equation}
with $c$ defined in \eqref{def:Phi}. Hence, we can expect that away from the Chebotarev compact $\Gamma$ the solution $\bm N$ of 
(RH-N1)--(RH-N3) models the behavior of $\bm T$ for $n$ large enough. 

We build  $\bm N$ in the following form, 
\begin{equation}
\label{defNouter}
\bm N(z):=F(\infty)^{\sigma_3} \widetilde{\bm N}(z) F(z)^{-\sigma_3},
\end{equation}
using two ``ingredients'' described in detail below: a scalar function $F$, which plays the role of a Szeg\H{o} function with piece-wise constant boundary values, and a matrix-valued function $\widetilde{\bm N}=(\widetilde{\bm N}_{ij})$, which will be defined in terms of abelian integrals on $\mathcal R$. Both $F$ and $ \widetilde{\bm N}$ depend on $n$, but in this section we omit this dependence from the notation, keeping it in mind. 
 

With the notation \eqref{def:holomorphicdiff} consider the equation 
$$
\int_{\bm \infty^{(1)}}^{\bm z_n} d\nu_0  =- \left(1+ \frac{1}{2\pi i}\, \log\left( \frac{s_{n,2}}{s_{n,1}}\right) \right)\, \oint_{\mathfrak a} d\nu_0 - \frac{1}{2\pi i}\, \, \log\left( \frac{s_{n,2}}{s_{n,3}}\right) \oint_{\mathfrak b} d\nu_0 ,$$
or equivalently,
\begin{equation}
\label{mainEqparameters}
\int_{\bm \infty^{(1)}}^{\bm z_n} d\nu_0^*   = -2\pi i \tau \left(1+ \frac{1}{2\pi i}\, \log\left( \frac{s_{n,2}}{s_{n,1}}\right) \right)  -   \log\left( \frac{s_{n,2}}{s_{n,3}}\right)  ,
\end{equation}
where $\nu_0^*$ is the normalized differential of the first kind \eqref{normalized}, and the path of integration lies entirely in the rectangle $\widetilde{\mathcal R}$. Among all possible choices of the branch of the logarithm, there is at most one  value of $\log(s_{n,2}/s_{n,1})$ and  at most one  value of $\log(s_{n,2}/s_{n,3})$ such that this equation has a solution in $\widetilde{\mathcal R}$; this solution $ {\bm z_n=(z_n,w_n)}$ is obviously unique. 
 
\begin{remark}
If $\bm z_n$ falls on one of the cycles $\Gamma_j$ we consider it slightly deformed so that the same argument applies. A truly  special situation occurs when eventually $\bm z_n=\infty^{(1)}$ or  $\bm z_n=\infty^{(2)}$. The first case happens  when 
\begin{equation}
\label{pathology}
  \frac{M_1^0}{2\pi i}\, \log\left( \frac{s_{n,2}}{s_{n,1}}\right)   + \frac{M_3^0}{2\pi i}\, \, \log\left( \frac{s_{n,2}}{s_{n,3}}\right) \equiv 0 \mod \Z.
\end{equation}
The consequences of this degeneration are discussed below, see Remark \ref{remark:exceptional}.
\end{remark}

With this choice of the branch of the value of $\log(s_{n,2}/s_{n,1})$ we define two parameters, $\beta _1$ and $\beta _2$, as follows:
\begin{align}
\label{def:theta1}
{\beta _1}:= & \log s_{n,1} =\log(t_1) + 2\pi i n(m_3-m_2). 
\\
\label{def:theta2}
{\beta _2}:= & \pi i + \log\left( \frac{s_{n,2}}{s_{n,1}}\right).
\end{align}
Obviously, \eqref{def:theta1} defines $\beta _1$ up to an additive constant which is an integer multiple of $2\pi i$.

With these two complex constants  fixed, we build a complex-valued function $F$ on $\C\setminus \Gamma$,  holomorphic, uniformly bounded and non-vanishing in $\overline\C\setminus \Gamma$, and such that
\begin{equation}
\label{Fbdryvalues}
F_+(z)F_-(z)=  \begin{cases} e^{\beta _1}=s_{n,1}, & z \in \Gamma_1\setminus \{v, a_1\}, \\
e^{ \beta _1+\beta _2}=-s_{n,2}, & z \in \Gamma_2\setminus \{v, a_2\}, \\
e^{\beta _1+\beta _3}=s_{n,1}e^{ \beta _3}, & z \in \Gamma_3\setminus \{v, a_3\}. \\
\end{cases}
\end{equation}
Constant $\beta _3$ is not arbitrary:
\begin{equation}
\label{theta3}
\beta _3= \left(1+\tau \right) \beta _2,
\end{equation}
with $\tau$ from \eqref{defTau}. 
We take $F$ of the form $F(z)=\exp({\Lambda(z)})$, and give two equivalent expressions for $\Lambda$. 

First, $\Lambda$ can be built in terms of the holomorphic differential $\nu_0$ on $\R$:
\begin{equation}
\label{LambdaDifferential}
\Lambda(z) =\frac{\beta _1}{2}+ \Xi\, \int_{a_1}^{\bm z^{(1)}} d\nu_0=\frac{\beta _1}{2}+ \Xi\, \int_{a_1}^{\bm z^{(1)}} \frac{dt}{w(t)}, \quad \bm z^{(1)}\in \mathcal R^{(1)}, \quad z\in \overline \C \setminus \Gamma,
\end{equation}
where
\begin{equation}
\label{def:Xi}
 {\Xi } =\beta _2 \left[ \left(M_1^2 - \tau M_3^2 \right) - \mathcal S \left(M_1^1 - \tau M_3^1 \right) \right] ,
\end{equation}
with $M_j^k$ introduced in \eqref{defMs} and $\mathcal S$ in \eqref{def:average}. 
The path of integration in \eqref{LambdaDifferential} lies entirely in $\mathcal R^{(1)}$, except for its initial point. Observe that $\Lambda$ in \eqref{LambdaDifferential} is a holomorphic function in $\overline \C \setminus \Gamma$.

Alternatively, define the functions
\begin{equation} \label{defEll}
\mathfrak l_j(z):=\frac{w(z)}{2\pi i}\,\int_{\Gamma_j} \frac{dt}{w_+(t) (t-z)}, \quad z\in \C\setminus \Gamma, \quad j=1,2, 3,
\end{equation}
where we integrate in the direction ``from $a_j$ to $v$'', and let
\begin{equation}
\label{defLambdaCauchy}
\Lambda(z)= \frac{\beta _1}{2}- \beta _2\, \big( \mathfrak l_1(z)- \tau  \mathfrak l_3(z) -1/ 2 \big), \quad z\in \C\setminus \Gamma.
\end{equation}

\begin{lemma}\label{lemma:Lambda}
With $\Lambda$ given either by  \eqref{LambdaDifferential}--\eqref{def:Xi} or by \eqref{defLambdaCauchy}, function  $F(z)=F_n(z)=\exp(\Lambda(z))$ is holomorphic, uniformly bounded and non-vanishing in $\overline\C\setminus \Gamma$, with
\begin{equation} \label{FatInfinity}
F(\infty)=\exp\left(\frac{\beta _1}{2}- \beta _2\, \big( M_1^1- \tau  M_3^1  -1/ 2 \big)\right).
\end{equation}	
Moreover, $F$ has continuous boundary values at $\Gamma^\circ$ that satisfy \eqref{Fbdryvalues}--\eqref{theta3}.

In consequence, formulas \eqref{LambdaDifferential}--\eqref{def:Xi} and \eqref{defLambdaCauchy} define the same function in $\overline\C\setminus \Gamma$. 
\end{lemma}
\begin{remark}
Recall that $\beta _2$ was defined uniquely as a function of $n$, but $\beta _1$ is determined mod $(2\pi i)$. From \eqref{LambdaDifferential}--\eqref{def:Xi} or \eqref{defLambdaCauchy} it follows that for each $n$, function $F$ is determined uniquely up to a change of sign.
\end{remark}

Now we define in $\C\setminus \Gamma$ the analytic matrix valued function $\widetilde{\bm N}$ entry-wise in terms of meromorphic differentials on $\mathcal R$ as follows (see also \cite{Kuijlaars:2011fk}). 
The meromorphic differential 
$$
 { d\eta^* (z)}= -\frac{1}{4} \frac{(AV)'(z)}{(AV)(z)} \, dz -\frac{1}{2} d\nu_0^* = \frac{1}{4}\left(-   \frac{1}{z-v} -  \sum_{j=0}^3 \frac{1}{z-a_j} +\frac{1}{ M_3^0\,  w(\bm z)}\right) dz
$$ 
has  only simple poles on $\mathcal R$: at the zeros of $AV$  with residues $-1/2$, and at $\infty^{(1)}$,  $\infty^{(2)}$, both with residues $+1$; additionally, its $\mathfrak b$-period  is zero.

With  $\bm z_n=(z_n,w_n)$ solving \eqref{mainEqparameters} we  consider the meromorphic differential
\begin{equation} \label{mainDiff1}
 {\eta_{\bm z_n}} = \eta^* +  \Omega_{\bm z_n,\infty^{(1)}},
\end{equation}	
or more explicitly,
\begin{equation} \label{mainDiff2}
 d \eta_{\bm z_n} =  \left(-    \frac{(AV)'(z)}{4 (AV)(z)} + \frac{1}{2 (z- z_n)}+  \frac{w_n}{2 w(\bm z) (z- z_n)} + \frac{z}{2  w(\bm z)} +\frac{\delta_1}{  w(\bm z)} \right) dz,
\end{equation}	
where $\delta_1$ is uniquely determined by the condition that the $\mathfrak b$-period of $\eta_{\bm z_n} $ is zero. Obviously, it has the only poles, all simple, at the zeros of $AV$ with residues $-1/2$, and at $\bm z_n$ and  $\infty^{(2)}$, both with residues $+1$.
\begin{lemma}
\label{lemma:period}
With the conditions above,
 \begin{equation}
\label{Abel2bis}
\oint_{\mathfrak a}  d \eta_{\bm z_n} = 
\beta _1+\beta _3 - \log(s_{n,3})   \mod (2\pi i),
\end{equation}
with $\beta_1$ and $\beta _3$ given in \eqref{def:theta1} and \eqref{theta3}, respectively.
\end{lemma}
\begin{proof}
Direct calculation shows that
$$
\oint_{\mathfrak a}  d\eta^*= \pi i \left( 1-\tau\right),
$$
and from the Riemann's identities it follows that
\begin{equation}
\label{Abel}
\oint_{\mathfrak a} \Omega_{\bm z_n,\infty^{(1)}}= \frac{1}{2 M_3^0}   \int^{\bm z_n}_{\infty^{(1)}} d\nu_0= -   \int^{\bm z_n}_{\infty^{(1)}} d\nu_0^*.
\end{equation}
It remains to use \eqref{mainEqparameters} and the definition of $\beta _1$ and $\beta _2$ above.
\end{proof}
\begin{remark}
The uniqueness of $\bm z_n$ satisfying \eqref{Abel2bis} can be easily established:  for any other $\bm r \in \mathcal R$, $\nu_{\bm z_n}-\nu_{\bm r}$ is a meromorphic differential in $\mathcal R$ whose only poles (both simple) are at $\bm z_n$ (with residue $1$) and $\bm r $ (with residue $-1$), and with periods multiple of $2\pi i$, so that $\exp(\int^z d(\nu_{\bm z_n}-\nu_{\bm r}))$ has a single pole at $\bm r$, which is impossible. 
\end{remark}

Let
\begin{equation}
\label{def:v}
 {u_1(z)}=\exp\left(\int_{\infty^{(1)}}^{\bm z^{(1)}} d\eta_{\bm z_n}\right), \quad  {u_2(z)}=\exp\left(\int_{\infty^{(1)}}^{\bm z^{(2)}} d\eta_{\bm z_n}\right), 
\end{equation}
with $\quad \bm z^{(j)}=\pi^{-1}(z)\cap \mathcal R^{(j)}$. For $u_1$, the path of integration lies entirely in $ \mathcal R^{(1)}$, while for $u_2$ it goes from $\mathcal D_\pm^{(1)}$ into $\mathcal D_\mp^{(2)}$, crossing $\Gamma_2$ once, see Figure \ref{fig:rectanglepaths}.

\begin{figure}[h]
\centering \begin{overpic}[scale=0.9]%
{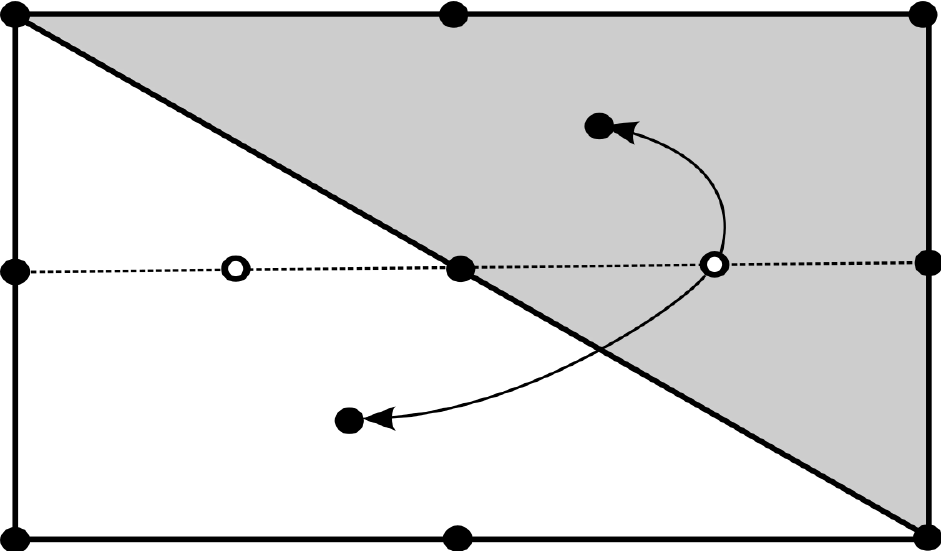}%
\put(61.5,48){\small $\bm z^{(1)}$}
\put(32,16){\small $\bm z^{(2)}$}
\put(-1.5,-2){\small $v$}
\put(-1,59){\small $v$}
\put(99,-2){\small $v$}
\put(99,59){\small $v$}
\put(-3.5,29){\small $a_1$}
\put(100,30){\small $a_1$}
\put(47,-2){\small $a_3$}
\put(47,59.5){\small $a_3$}
\put(48,26){\small $a_2$}
\put(23,26){\small $\infty^{(2)}$}
\put(74,26){\small $\infty^{(1)}$}
\put(7,40){\small $\mathcal D_-^{(2)}$}
\put(7,10){\small $\mathcal D_+^{(2)}$}
\put(87,20){\small $\mathcal D_-^{(1)}$}
\put(87,45){\small $\mathcal D_+^{(1)}$}
\put(70,38){\small $u_1$}
\put(49,18.5){\small $u_2$}
\put(77,15){\small $\Gamma_{2+}$}
\put(25,45){\small $\Gamma_{2-}$}
\end{overpic}
\caption{Paths of interration for functions $u_j$ defined in \eqref{def:v}.}
\label{fig:rectanglepaths}
\end{figure}
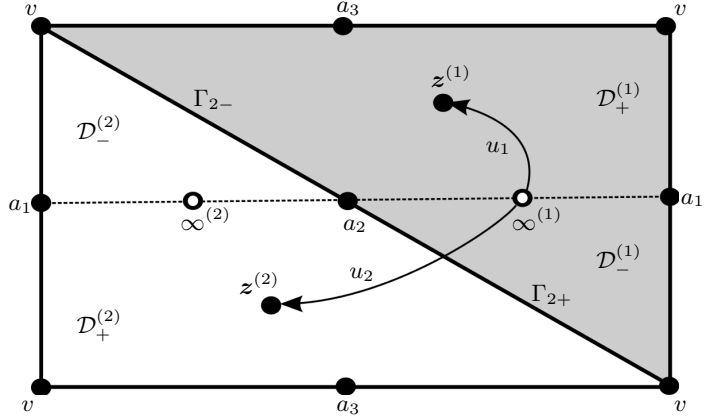

\begin{lemma}
\label{lemma:bdryvaluesV}
Functions $u_j$ are holomorphic in $\overline\C\sm (\gamma_1^\perp \cup \Gamma  \cup \gamma_2^\perp)$ (see Figure~\ref{fig:Gamma}), and have a continuous boundary values on $\gamma_1^\perp \cup \Gamma^\circ  \cup \gamma_2^\perp$ such that
\begin{equation} \label{bdryvaluesV}
(u_1)_+(z)= \begin{cases}
(u_1)_-(z), \quad  z\in \gamma_1^\perp    \cup \gamma_2^\perp,\\
(u_2)_-(z), \quad  z\in \Gamma_1^\circ \cup \Gamma_2^\circ, \\
- \dfrac{e^{\beta _1+\beta _3}}{s_{n,3}}(u_2)_-(z), \quad  z\in \Gamma_3^\circ;
\end{cases} \quad 
(u_2)_+(z)= \begin{cases}
-(u_2)_-(z), \quad   z\in \gamma_1^\perp    \cup \gamma_2^\perp, \\
(u_1)_-(z), \quad  z\in \Gamma_1^\circ \cup \Gamma_2^\circ, \\
 \dfrac{s_{n,3}}{e^{\beta _1+\beta _3}} (u_1)_-(z), \quad  z\in \Gamma_3^\circ.
\end{cases}
\end{equation}	
Moreover, as $z\to a_j$, $z \in \mathbb C \setminus \Gamma$, $j=1, 2, 3$, $
u_k(z)= \mathcal O(|z-a_j|^{-1/4})$, while as $z\to v$, $z \in \mathbb C \setminus \Gamma$, $
u_k(z)= \mathcal O(|z-v|^{-1/4})$. 
Additionally,
$$
u_1(z)=1+\mathcal O\left( \frac{1}{z}\right),  \quad u_2(z)= \mathcal O\left( \frac{1}{z}\right), \quad z\to \infty. 
$$
Finally, if $\mb z_n \in \mathcal R^{(1)}$ then $u_1$ has a simple zero at $z=z_n$ and $ u_2(z_n)\neq 0$; otherwise, $u_2$ has a simple zero at $z=z_n$ and $ u_1(z_n)\neq 0$.
\end{lemma}

With these two functions we define in $\C\setminus \Gamma$
\begin{equation} \label{defNtilde}
\mathbf{\widetilde N}_{11}(z) =u_1(z), \quad \mathbf{\widetilde N}_{12}(z)=\begin{cases} u_2(z), & \text{if } z\in D_+ , \\
-u_2(z), & \text{if } z \in D_-.
\end{cases}
\end{equation}	
Recall that the simply connected domains $D_\pm$ are limited by $\gamma_1^\perp \cup \Gamma  \cup \gamma_2^\perp$, and  ${D_+}$ is the one containing $a_3$ on its boundary.
\begin{remark}
\label{remark:exceptional}
As we have noticed before, it may happen that either $\bm z_n=\bm \infty^{(1)}$ or $\bm z_n=\bm \infty^{(2)}$. In the first case, when condition \eqref{pathology} holds, we just have ${\eta_{\bm z_n}} = \eta^*$, so that $\bm N_{11}$ has a zero at infinity. This, as it follows from the asymptotic formulas below, will mean that the index $n$ is not normal, see the expression of $\chi$ in \eqref{def:chi}. In the second case,  ${\eta_{\bm z_n}} ={\eta_{\bm \infty^{(2)}}} $ has a simple pole at $\infty^{(2)}$  with residue $+2$, which creates a double zero of  $\bm N_{12}$ at infinity. 
\end{remark}

Furthermore, consider a family of functions $q$ on $\mathcal R$ of the form
$$
q(\bm z)= a+b\left( \frac{w(\bm z)+w_n}{z-z_n}-z\right),  \quad a,b\in \C.
$$
Each such a function has a simple pole at $\bm z_n=(z_n, w_n)$ and at $\mb \infty^{(2)}$. There is a unique combination of constants $a$, $b$, 
such that additionally
$$
q(\infty^{(1)})=0 \quad \text{and}\quad \lim_{\bm z\to \infty^{(2)}} q(\bm z)\widetilde{\bm N}_{12}(z)=1.
$$

Let $q^{(j)}(z)=q(\bm z^{(j)})$ be the values of $q$ on the $j$-th sheet. Then set
$$
\widetilde{\bm N}_{21}(z)= q^{(1)}(z) \widetilde{\bm N}_{11}(z), \quad \widetilde{\bm N}_{22}(z)= q^{(2)}(z) \widetilde{\bm N}_{12}(z).
$$
This defines completely the matrix-valued function $\widetilde{\bm N}=(\widetilde{\bm N}_{ij})$.  
Finally, we assembly ${\bm N}$ as in \eqref{defNouter} using this matrix $\widetilde{\bm N}$ and function $F$ given by \eqref{LambdaDifferential}--\eqref{def:Xi} or \eqref{defLambdaCauchy} with parameters \eqref{def:theta1}--\eqref{def:theta2}. Direct verification shows that the following statement holds true:
\begin{proposition}\label{prop:global}
Matrix ${\bm N}$ constructed above solves the RH problem (RH-N1)--(RH-N3).
\end{proposition}

\subsection{Local parametrix} 

Matrix $\bm T$ defined in \eqref{def:matrixT} has jumps that are asymptotically close to the identity matrix for $n$ large enough, as long as we stay away from $\Gamma$. However, this behavior fails in a neighborhood of $\mathcal A$ and the Chebotarev center $v$, where we need to perform a separate analysis in order to find an appropriate model. Here we describe only the construction of the local parametrix $\bm P$ at $z=v$ (around the  branch points $a_j$ matrix $\bm P$ is built in the way described in detail in \cite{MR2087231}). 

We take a small $\delta>0$ and define $ {D_\delta}:=\{z\in \C:|z-v|<\delta\}$, $ {B_\delta}:=\{z\in \C:|z-v|=\delta\}$, assuming that $D_\delta \cap \mathcal A =\emptyset$, see Figure~\ref{fig:Local}.

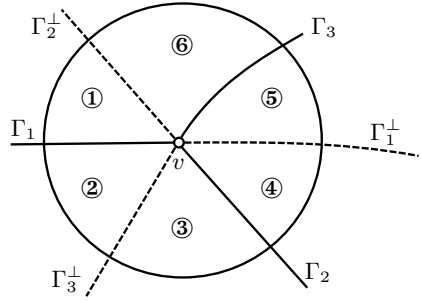
\begin{figure}[h]
\centering \begin{overpic}[scale=1.1]%
{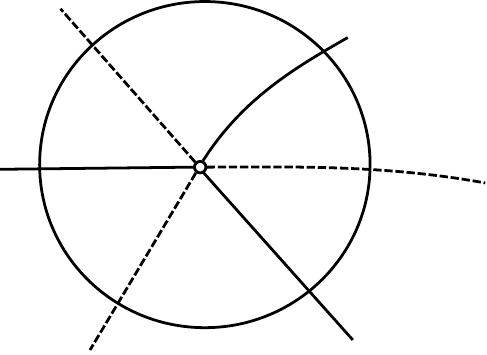}%
\put(39.5,32){\small $v$}
\put(0,40){\small $\Gamma_1$}
\put(88,39){\small $\Gamma_1^\perp$}
\put(72,5){\small $\Gamma_2$}
\put(5,65){\small $\Gamma_2^\perp$}
\put(73.5,65){\small $\Gamma_3$}
\put(10,3){\small $\Gamma_3^\perp$}
    \put(15,47){  \large \ding{172}  }
 \put(15,25){  \large  \ding{173}  }
 \put(37,15){  \large  \ding{174}  }
    \put(37,60){  \large \ding{177}  }
 \put(59,47){  \large  \ding{176}  }
 \put(59,25){  \large  \ding{175}  }
\end{overpic}
\caption{Local parametrix.}
\label{fig:Local}
\end{figure}

 The local parametrix at $z=v$ has the form
 \begin{equation}
\label{P}
 {\textbf P(z)}:=\bm E(z) \bm \Psi\left(n^{2/3} \varphi(z) \right) b^{\sigma_3} \bm B \, \Phi^{-n\sigma_3}(z), \quad z\in D_\delta \setminus (\Gamma\cup \Gamma^\perp). 
\end{equation}
Here $\Phi$ is the conformal mapping defined in \eqref{defPhiconformal}, and the rest of the ingredients are:
\begin{itemize}
\item Function $\varphi$, defined piece-wise in each sector: for $z\in D_\delta \setminus (\Gamma\cup \Gamma^\perp)$,
\begin{equation} \label{conformalLocal}
\varphi(z) :=\begin{cases} \displaystyle \left( \frac{3}{2}\int_{v}^z T(t) dt \right)^{2/3} & \text{if } z\in \text{\ding{175}} \cup \text{\ding{176}},   \\
\displaystyle \left(- \frac{3}{2}\int_{v}^z T(t) dt \right)^{2/3} & \text{otherwise, }
\end{cases}
\end{equation}	
where we take the main branch of the power function. Then $\varphi$ is a conformal mapping of $D_\delta$ onto a neighborhood of the origin, $\varphi(v)=0$, and $\Gamma_1^\perp$ is mapped onto the positive semi axis. 

\item Constants $ {b_1}$ and  $ {b_2}$, determined up to a change of sign by
$$
b_1 b_2=1/t_1, \quad b_1 /b_2=-t_2/t_3,
$$
and 
$$
b:=\begin{cases}
b_1, & \text{if } z \in  \text{\ding{172}} \cup \text{\ding{176}} \cup \text{\ding{177}}, \\
b_2, & \text{if } z \in  \text{\ding{173}} \cup \text{\ding{174}} \cup \text{\ding{175}}. \\
\end{cases}
$$
\item Matrices
$$
 {\bm B}:=\begin{cases}
\textbf I, & \text{if } z \in  \text{\ding{172}} \cup \text{\ding{173}} \cup \text{\ding{174}} \cup \text{\ding{177}}, \\
\begin{pmatrix} 0 & -t_2 \\ 1/t_2 & 0
\end{pmatrix}
, & \text{if } z \in     \text{\ding{175}},\\
\begin{pmatrix} 0 & t_3 \\ -1/t_3 & 0
\end{pmatrix}
, & \text{if } z \in     \text{\ding{176}},\\
\end{cases}
$$
and
$$
 {\bm M_2}=\bm M_2(z,n) :=\begin{cases}
(\kappa_3^n/b_1)^{\sigma_3} , & \text{if } z \in  \text{\ding{172}} \cup  \text{\ding{177}}, \\
(\kappa_2^n/b_2)^{\sigma_3}, & \text{if } z \in  \text{\ding{173}} \cup  \text{\ding{174}},\\
\begin{pmatrix} 0 & b_2 t_2 \\ -1/(b_2 t_2) & 0
\end{pmatrix}
, & \text{if } z \in z \in  \text{\ding{175}} \cup  \text{\ding{176}}.\\
\end{cases}
$$
\item Matrix valued functions in $D_\delta\setminus \Gamma_1$,
\begin{equation*} \label{defMatrixM1E}
 {\bm M_1(z)}:=\sqrt{\pi}\, \begin{pmatrix}
e^{\pi i/6} & -e^{\pi i/6} \\
e^{-\pi i/3} & e^{-\pi i/3}
\end{pmatrix}\, \varphi^{\sigma_3/4}(z),
\end{equation*}	
where we take the main branch of the root, and
\begin{equation} \label{defMatrixE}
 {\bm E(z)}=\bm E(z,n):= \bm N(z) \bm M_2(z) \bm M_1(z) n^{\sigma_3/6}, \quad z\in D_\delta\setminus \Gamma,
\end{equation}	
with $\bm N$ constructed in subsection \ref{sec:outer}, see \eqref{defNouter}. $\bm E(z)$ extends in fact as a  holomorphic function to the whole $D_\delta$.
\item The Airy parametrix, $ {\bm \Psi}$, defined as
\begin{align*}
\bm \Psi(\zeta):= & \begin{pmatrix} Ai(\zeta) & Ai(\omega^2 \zeta) \\
Ai'(\zeta) & \omega^2Ai'(\omega^2 \zeta)
\end{pmatrix} \omega^{-\sigma_3/4}, \quad \zeta \in \varphi^{-1}(\text{\ding{176}} \cup  \text{\ding{177}}), \\
\bm \Psi(\zeta):= & \begin{pmatrix} Ai(\zeta) & Ai(\omega^2 \zeta) \\
Ai'(\zeta) & \omega^2Ai'(\omega^2 \zeta)
\end{pmatrix} \omega^{-\sigma_3/4}\, \begin{pmatrix} 1 & 0 \\ -1 & 1
\end{pmatrix}
, \quad \zeta \in \varphi^{-1}(\text{\ding{172}}  ), \\
\bm \Psi(\zeta):= & \begin{pmatrix} Ai(\zeta) & -\omega^2 Ai(\omega \zeta) \\
Ai'(\zeta) & - Ai'(\omega \zeta)
\end{pmatrix} \omega^{-\sigma_3/4}, \quad \zeta \in \varphi^{-1}(\text{\ding{174}} \cup  \text{\ding{175}}), \\
\bm \Psi(\zeta):= & \begin{pmatrix} Ai(\zeta) & -\omega^2 Ai(\omega \zeta) \\
Ai'(\zeta) & - Ai'(\omega \zeta)
\end{pmatrix} \omega^{-\sigma_3/4}\, \begin{pmatrix} 1 & 0 \\  1 & 1
\end{pmatrix}, \quad \zeta \in \varphi^{-1}(\text{\ding{173}} ), 
\end{align*}
where $\omega=\exp(2\pi i/3)$, see e.g.~\cite{MR2001f:42037, MR2001g:42050}.
 \end{itemize}
 
 \begin{theorem}
\label{thm:P}
 Matrix-valued function $\mb P$ given by \eqref{P} solves the following boundary-value problem in $D_\delta$:
 \begin{enumerate}
\item[(RH-P1)] It has continuous boundary values $\bm P_\pm$ on both sides of all curves, and with the specified orientation,
$$
\bm P_+(z)=\bm P_-(z)\,\bm J_P (z), \quad z\in (\Gamma^\circ \cup \Gamma^\perp)\cap D_\delta, 
$$
where 
  \begin{align}\label{J_S3}
\bm J_P = &  \begin{pmatrix}
0 & s_{n,j} \\
-s_{n,j}^{-1} & 0 \end{pmatrix}, \quad \text{if } z \in \Gamma_j^\circ\cap D_\delta; \\
\label{J_S4}
= & \begin{pmatrix}
1 & 0 \\ 
 \frac{t_j }{t_{j-1} t_{j+1}}    \Phi ^{-2n}(z) & 1 \end{pmatrix} 
, \quad \text{if } z \in \Gamma_j^\perp \cap D_\delta. 
\end{align} 

\item[(RH-P2)] $\bm P(z)=  \left( \bm I +\mathcal O(1/n)\right) \mb N(z) $, for $z\in B_\delta$.
\item[(RH-P3)] As $z\to v$, $z \in \mathbb C \setminus ( \Gamma\cup \Gamma^\perp)$, 
$  \bm P(z)=  \mathcal O(1)$.
\end{enumerate}
\end{theorem}

\subsection{Asymptotic analysis} 

In the final transformation of the original problem (RH-Y1)--(RH-Y4) we define a
matrix valued function $\bm R$ in the form $\bm T(z) \bm A ^{-1}(z)$, where, roughly speaking, $\bm A = \bm N$ away from $\Gamma$ and $\bm A=\bm P$ in a neighborhood of $\mathcal A$ and $v$. The explicit formula for $\bm A$  in $D_\delta$ is given above, while it is built in terms of the Bessel functions in a neighborhood of the branch points $a_j$, see \cite{MR2087231} for details.  The inverses of all these matrices exist, since
    the determinants of these matrices are equal to $1$. 
    
The construction of $\bm N$ and $\bm P$ is such that
$$
 \bm  R_+(z) =  \bm  R_-(z) \left(\bm I + \mathcal O\left(\frac{1}{n} \right)\right), 
        \qquad  n\to\infty,
$$
uniformly on a finite set of contours in $\C$. Following the already standard reasoning we conclude that 
    \begin{equation} \label{AsympR}
      \bm  R(z) = \bm I + \mathcal O\left(\frac{1}{n} \right), 
        \qquad  n\to\infty,
    \end{equation}
uniformly in $\C$. 
    The relation (\ref{AsympR}) is the main term in
    the asymptotics for $\bm R$ and it is enough to give the leading term in the
    asymptotics for $\bm Y$. 
    

For instance, taking into account \eqref{def:matrixT} and \eqref{AsympR} we see that locally uniformly in $\C\setminus \Gamma$ we have
$$
\textbf Y(z)=c^{-n\sigma_3} \left( \bm I + \mathcal O\left( \frac{1}{n}\right) \right) \bm N(z) f(z)^{-\sigma_3/2}\Phi^{n\sigma_3}(z),
$$
where $c$ was defined in \eqref{def:Phi}.
In particular, working out the expressions for $\bm Y_{11}$ and $\bm Y_{12}$ we get:
\begin{theorem}\label{thm:exteriorAsympRH}
Locally uniformly in $\C\setminus \Gamma$,
\begin{equation}
\label{asymptoticsOutside}
\begin{split}
Q_n(z)
= \left( \frac{\Phi(z)}{c} \right)^n  \frac{F(\infty)}{f(z)^{1/2}(z)  F(z)} \left[ \exp\left(\int_{\infty^{(1)}}^{\bm z^{(1)}} d\eta_{\bm z_n}\right) \left( 1 + \mathcal O\left( \frac{1}{n}\right) \right)  +    \mathcal O\left( \frac{1}{n}\right)   \right],
\end{split}
\end{equation}
and
$$
  R_n(z)=    \frac{ f(z)^{1/2}}{\left( c \Phi(z) \right)^{n}}F(\infty) F(z) \left[ \pm \exp\left(\int_{\infty^{(1)}}^{\bm z^{(2)}} d\eta_{\bm z_n}\right) \left( 1 + \mathcal O\left( \frac{1}{n}\right) \right)  +    \mathcal O\left( \frac{1}{n}\right)    \right].
$$
\end{theorem}
The sign and the paths of integration are selected in accordance with the definition of $\widetilde{\bm N}_{1j}$ in \eqref{defNtilde}.

We see in particular, that the spurious zero of $Q_n$ is asymptotically close to the unique zero of $\bm N_{11}$, which appears only when $\bm z_n$ is on the first sheet. Otherwise, it gives us an extra interpolation condition (zero of $\bm Y_{12}$, close to the zero of $\bm N_{12}$).

The Riemann-Hilbert analysis yields  asymptotic formulas not only away from $\Gamma$ but in the rest of the regions. For instance, close to $\Gamma$ but still away from the branch points $a_j$ and the Chebotarev center $v$ the asymptotic expression for $Q_n$ is a combination of two competing terms, which gives rise to zeros of $Q_n$. For instance, by \eqref{def:matrixT} and \eqref{AsympR}, for $z $ in \ding{172} of the domain $D$, 
$$
\textbf Y(z)=c^{-n\sigma_3} \left( \bm I + \mathcal O\left( \frac{1}{n}\right) \right) \bm N(z) \begin{pmatrix}
1 & 0 \\
\frac{\tau_1}{t_1 \Phi^{2n}(z)} & 1
\end{pmatrix} f(z)^{-\sigma_3/2}\Phi^{n\sigma_3}(z).
$$
In particular,
\begin{align*}
c^n \,Q_n(z) f(z)^{1/2}= &     \left( \bm N_{11}(z) \Phi^n(z) + \bm N_{12}(z) \frac{\tau_1}{t_1}\Phi^{-n}(z) \right)\,  \left( 1 + \mathcal O\left( \frac{1}{n}\right) \right)\\
& + \left( \bm N_{21}(z) \Phi^n(z) + \bm N_{22}(z) \frac{\tau_1}{t_1}\Phi^{-n}(z) \right)\,    \mathcal O\left( \frac{1}{n}\right).  
\end{align*}
For $z\in \Gamma_1$ we can rewrite it as 
\begin{align*}
 Q_n(z) = & \left( \frac{\Phi _+(z)}{c}\right)^n  \left( \frac{\bm N_{11}(z)_+ }{f(z)_+^{1/2}} + \frac{\bm N_{11}(z)_-}{f(z)^{1/2}_-}    \right)\,  \left( 1 + \mathcal O\left( \frac{1}{n}\right) \right)  .
\end{align*}


Finally, in order to analyze the behavior at the Chebotarev center $v$, for instance, when $z\in D_\delta \cap \text{\ding{172}}$ (see Figure~\ref{fig:Local}), it is sufficient to obtain the expression for $\bm Y$ from
$$
\textbf Y(z)= c^{-n\sigma_3} \left( \bm I + \mathcal O\left( \frac{1}{n}\right) \right) \bm P(z) 
\begin{pmatrix}
1 & 0 \\
\frac{\tau_1}{t_1 \Phi^{2n}(z)} & 1
\end{pmatrix}
f(z)^{-\sigma_3/2}\Phi^{n\sigma_3}(z),
$$
where all the ingredients in the right hand side were given above. Since the formulas for $Q_n$ and $R_n$ obtained this way are not easily simplified, we omit their explicit calculation here for the sake of brevity.

\section{Wrapping up, or matching the asymptotic formulas and the Nutall's conjecture} \label{sec:comparison}

In \cite{MR769985} Nuttall  conjectured the form of the leading term of asymptotics for $Q_n$ and $R_n$ away from $\Gamma$ in terms of a solution of a  \emph{scalar} boundary value problem. We show next that our results match the Nuttall's conjecture.


Let us denote
\begin{equation}
\label{def:chi}
 {\chi(z)}:=\left( \frac{\Phi(z)}{c} \right)^n \frac{\bm N_{11}(z)}{f(z)^{1/2}}, \quad  {\mathfrak R}(z):=  \left( c \Phi(z) \right)^{-n}  \bm N_{12}(z)  f(z)^{1/2},
\end{equation}
so that $\chi$ has a pole of order $n$ at infinity, and ${\mathfrak R}$ has there a zero of order $n+1$ (unless the pathological situation of $z_n=\infty$ occurs). By \eqref{asymptoticsOutside},
$$
Q_n(z)= \chi(z)\left( 1 + \mathcal O\left( \frac{1}{n}\right) \right), \quad z\in \C\setminus \Gamma.
$$
But for $z\in \Gamma_j^\circ$,
$$
\left(f(z)\chi(z) \right)_\pm=f(z)^{1/2}_\pm \left( \frac{\Phi(z)_\pm}{c} \right)^n \bm N_{11\pm}(z);
$$
using that $\rho(z)=t_j f_+(z)/\tau_j=t_j \tau_j f_-(z)$ on $\Gamma_j^\circ$,  \eqref{boundaryforN} and Lemma \ref{lemma1}, we get
\begin{align*}
\sigma(z)  \chi_+ (z)=& -w_+(z)\frac{f(z)^{1/2}_+}{\tau_j} \left( c \,\Phi(z)_- \right)^{-n} \bm N_{12-}(z)=(w  \mathfrak R)_-(z),\\
\sigma(z)    \chi _- (z)= & w_+(z) \tau_j  f(z)^{1/2}_-  \left( c\, \Phi(z)_+  \right)^{-n}   \bm N_{12+}(z)= (w \mathfrak R)_+(z),
\end{align*}
where $\sigma(z):=\rho(z) w_+(z)$ on $\Gamma^\circ$.
These two equations match  the boundary value conditions in \cite[formula (5.6)]{MR891770} (after replacing $\chi_2=\chi$ and $H=w \mathfrak R$) that define uniquely the leading asymptotic terms for $Q_n$ and $R_n$, according to the conjecture of Nuttall.
 
Let us finally compare  the asymptotic formulas obtained in Sections \ref{sec:wkb} and \ref{sec:RH}, and given by Theorems \ref{thm:wkb} and \ref{thm:exteriorAsympRH}. We introduce here the notation
$$
u(z)=\int_{a_1}^z \frac{ dt}{w(t)} = \int_{a_1}^z d\nu_0.
$$
On one hand, observe that in the case $p=3$, with the function $\theta(\bm z,\bm \zeta)$ defined in \eqref{functionV} we have
$$
 \theta(\bm z_n,z)=\frac{w_n}{V(z_n)} u(z) + w_n \int_{a_1}^z \frac{dt}{w(t)(t-z_n)} ,
$$
so that 
\begin{align*}
N \mathcal H_n(z) = & N G(z,\infty)  +   d_{n}\, u(z)+\frac{1}{2}\,   \theta( \bm z_{n},z) \\
= & n \log \Phi(z) +\frac{w_n}{2} \int_{a_1}^z \frac{dt}{w(t)(t-z_n)} + \frac{1}{2}\int_{a_1}^z \frac{t\, dt}{w(t)}   +\delta_2  u(z) +\mathcal O\left( \frac{1}{n}\right),
\end{align*}
where
$$
\delta_2=d_n -\frac{v}{2}+\frac{w_n}{2 V(z_n)}.
$$
On the other hand, up to a multiplicative constant,
$$
F(z)^{-1}\exp\left(\int ^{\bm z^{(1)}} d\eta_{\bm z_n}\right) = \exp\left( \int ^{\bm z^{(1)}} d( \eta_{\bm z_n} -\Xi\, \nu_0) \right).
$$
Recalling the expression for $\eta_{\bm z_n} $ in \eqref{mainDiff2}, we get 
\begin{align*}
F(z)^{-1} & \exp\left(\int ^{\bm z^{(1)}} d\eta_{\bm z_n}\right)   = \\
& = \frac{(z-z_n)^{1/2}}{(AV)^{1/4}(z)} \exp\left( \int ^{\bm z^{(1)}} \left( \frac{w_n}{2} \int_{a_1}^z \frac{dt}{w(t)(t-z_n)} + \frac{1}{2}\int_{a_1}^z \frac{t\, dt}{w(t)}    \right) dt + \delta_3 u(z) \right),
\end{align*}
with an appropriate selection of the constant $\delta_3$. Comparing thus expressions \eqref{Hcal}--\eqref{WKBforQn}, obtained by the WKB analysis, with \eqref{asymptoticsOutside} we see that they coincide, up to the right determination of the constants $\delta_2$ and $\delta_3$ above. But these constants are uniquely determined by the condition that the right hand side in \eqref{WKBforQn} and in \eqref{asymptoticsOutside} must be single-valued in $\C\setminus \Gamma$.




\def\cprime{$'$}
\providecommand{\bysame}{\leavevmode\hbox to3em{\hrulefill}\thinspace}
\providecommand{\MR}{\relax\ifhmode\unskip\space\fi MR }
\providecommand{\MRhref}[2]{%
  \href{http://www.ams.org/mathscinet-getitem?mr=#1}{#2}
}
\providecommand{\href}[2]{#2}

\end{document}